\documentclass[11pt, a4paper]{amsart}
\usepackage{color}
\usepackage{graphicx}

\def\N{{\mathcal{N}}}
\def\ep{\varepsilon}
\def\O{\Omega}
\def\di{\displaystyle}

\def\R{{\mathbb {R}}}

\def\A{{\mathcal{A}}}

\def\lp{L^{p(\cdot)}(\Omega)}
\def\lq{L^{q(\cdot)}(\Omega)}

\def\lpe{L^{p^*(\cdot)}(\Omega)}
\def\wp{W^{1,p(\cdot)}(\Omega)}
\def\wpgama{{W_{\Gamma_D}^{1,p(\cdot)}(\Omega)}}

\def\lqb{L^{q(\cdot)}(\partial\Omega)}
\def\lpbg{L^{p(\cdot)}(\Gamma_D)}

\def\esssup{\text{esssup}}
\def\essinf{\text{essinf}}

\def\th{\mathcal{T}_h}
\def\thp{\mathcal{T}_{h'}}
\def\eh{\mathcal{E}_h}
\def\nh{\mathcal{N}_h}
\def\wpt{W^{1,p(\cdot)}(\th)}

\def\gi{\Gamma_{int}}
\def\lpi{L^{p(\cdot)}(\Gamma_{int})}
\def\lpd{L^{p(\cdot)}(\Gamma_{D})}
\def\lpk{L^{p(\cdot)}(\kappa)}

\def\[{[\hspace{-0.05cm}[}
\def\]{]\hspace{-0.05cm}]}

\newtheorem{teo}{Theorem}[section]
\newtheorem{lema}[teo]{Lemma}
\newtheorem{prop}[teo]{Proposition}
\newtheorem{corol}[teo]{Corollary}
\newtheorem{hyp}[teo]{Hypothesis}
\theoremstyle{definition}
\newtheorem{defi}[teo]{Definition}

\theoremstyle{remark}
\newtheorem{remark}[teo]{Remark}

\newtheorem*{ack}{Acknowledgements}

\def\ess{\mathop{\mbox{\normalfont ess}}\nolimits}

\title{ Interior penalty discontinuous Galerkin FEM for the $p(x)$-Laplacian.}

\author[L. M. Del Pezzo]{Leandro M. Del Pezzo} \address{Departamento de Matem\'atica,
Facultad de Ciencias Exactas y Naturales, Universidad de Buenos
Aires, 1428 Buenos Aires, Argentina.} \email{ldpezzo@dm.uba.ar}
\author[A. L. Lombardi]{Ariel L. Lombardi} \address{Instituto de Ciencias, Universidad Nacional de
General Sarmiento, J.M. Gutierrez 1150, B1613GSX Los Polvorines,
Provincia de Buenos Aires, and Departamento de Matem\'atica,
Facultad de Ciencias Exactas y Naturales, Universidad de Buenos
Aires, 1428 Buenos Aires, Argentina.} \email{aldoc7@dm.uba.ar} 
\author[S. Mart\'\i nez]{Sandra Mart\'{i}nez} \address{Departamento de Matem\'atica,
Facultad de Ciencias Exactas y Naturales, Universidad de Buenos
Aires, 1428 Buenos Aires, Argentina.} \email{smartin@dm.uba.ar}
\thanks{Supported by ANPCyT under grants PICT 2006-290 and PICT 2010-1675, and by Universidad de Buenos Aires under grants X070, X078 and X117 . All three authors are members of CONICET Argentina}

\begin{document}


\maketitle
\begin{abstract}
In this paper we construct an ``Interior Penalty" Discontinuous
Galerkin method to approximate the minimizer of a variational
problem related to the $p(x)-$Laplacian. The function $p:\O\to
[p_1,p_2]$ is log H\"{o}lder continuous and $1<p_1\leq
p_2<\infty$. We prove that the minimizers of the discrete
functional converge to the solution. We also make some numerical
experiments in dimension one to compare this method with the
Conforming Galerkin Method, in the case where $p_1$ is close to
one. This example is motivated by its applications to image
processing.
\end{abstract}

\section{Introduction}

In this paper we study a discontinuous Galerking method to
approximate the minimizer of a non homogeneous functional. This
functional is related to the so-called  $p(x)-$Laplacian operator,
i.e.
\begin{equation}\label{px}
\Delta_{p(x)}u=\mbox{div}(|\nabla u(x)|^{p(x)-2}\nabla u).
\end{equation}

This operator extends the classical Laplacian ($p(x)\equiv 2$) and
the so-called $p-$Laplacian ($p(x)\equiv p$ with $1<p<\infty$) and
it has been recently used in image processing and in the modeling
of electrorheological fluids, see \cite{BCE, CLR, R} .

\medskip
In an image processing problem, the aim is to recover the real
image $I$ from an observed image $\xi$ of the form $\xi=I+\eta$,
where $\eta$ is a noise.

Approaches to image denoising have been developed along three
main lines: wavelet methods, stochastic methods and variational
methods, see references in \cite{BCE}. One variational approach
that has attracted a great deal of attention is the total
variation method of L. Rudin,  S. Osher and E. Fetami \cite{ROF}.
The variational problem is

\medskip
Minimize the functional $|Du|(\O)$ over all the functions in
$BV(\O)\cap L^2(\O)$ such that
$$
\int_{\O} u\, dx =\int_{\O} \xi\, dx\quad \mbox{ and }\quad
\int_{\O} |u-\xi|^2\, dx =\sigma^2$$ for some $\sigma>0$.

 \medskip
The conditions on the space come from the assumption that $\xi$ is
a function that represents a white noise with mean zero and
variance $\sigma$. Moreover, the authors prove that this problem
is equivalent to minimizing
$$|Du|(\O)+\frac{\lambda}{2} \int_{\O}|u-\xi|^2\, dx$$
for some nonnegative Lagrange multiplier
$\lambda=\lambda(\sigma,\xi)$. This model works when the image is
piecewise constant, but in some cases can cause a {\it
staircasing} effect. See for instance \cite{CL}.

An older approach consists in minimizing,
$$\int_{\O}|\nabla u|^2+\frac{\lambda}{2} \int_{\O}|u-\xi|^2\, dx.$$
This method solves the {\it staircasing} effect, but it has the
problem that it does not preserve edges.

\medskip
 In \cite{CLR}, the authors introduce a model that involves the $p(x)-$Laplacian, for some function  $p:\O\to [p_1,2]$, with $p_1>1$.
 This function encodes the information on the regiones where the gradient is sufficiently large (at edges) and where the gradient is close to zero (in homogeneous regions).
In this manner, the model avoids the {\it staircasing} effect
still preserving the edges.

Recently, in \cite{BCE}  the authors propose a variant of the
method of Chen, Levine and Rao \cite{CLR}. More precisely, they
consider the functional
$$\int_{\O}|\nabla u|^{p(x)}+\frac{\lambda}{2} \int_{\O}|u-\xi|^2\, dx,$$
with $p:\O\to [1,2]$ a function such that
 $p(x)=P_M(|\nabla G_{\delta}*\xi|(x)),$ where $G_{\delta}(x)$ is an approximation of the identity, $M>>1$ and $P_M$ is a function that satisfies
$P_M(0)=2$ and $P_M(x)=1$ for all $|x|>M$. Observe that, since
$p(x)=1$ for some values of $x$,  the authors have to rewrite the
functional in a form that
 allow for computation of weak derivatives.

\bigskip

Motivated by the above mentioned applications, we  study a
numerical method to approximate minimizers of a functional related
to the  $p(x)-$ Laplacian.

We work in the following setting:

\medskip
Let $\Omega$ be a bounded Lipschitz domain. For functions
$p,s,t$ the following conditions will be assumed when necessary,
\begin{itemize}
\item[(H1)] $p:\overline{\O}\to [p_1,p_2]$ ($1<p_1\leq p_2
<\infty$)  is $\log$-H\"{o}lder continuous. That is, there exists
a constant $C_{log}$ such that
$$|p(x) - p(y)| \leq \frac{C_{log}}{\log\,\left(e+\frac{1}{|x - y|}\right)}\quad \forall\, x,y\in\O;$$
\item[(H2)]  $s\in L^{\infty}(\O),$ with $1\leq s(x)< p^*(x)-\ep$
for some $\ep>0;$ \item[(H3)]  $t\in C^0(\partial\O)$ with $1\leq
t(x)< p_*(x).$
\end{itemize}
Here, $p^*$ and $p_*$ are the Sobolev critical exponents for these
spaces, i.e.
\begin{equation}\label{criticos}
p^*(x):=
\begin{cases}
\frac{p(x)N}{N-p(x)} & \mbox{if } p(x)<N,\\
+\infty & \mbox{if } p(x)\ge N,
\end{cases}
\quad \mbox{ and }\quad p_*(x):=
\begin{cases}
\frac{p(x)(N-1)}{N-p(x)} & \mbox{if } p(x)<N,\\
+\infty & \mbox{if } p(x)\ge N.
\end{cases}
\end{equation}

\medskip

Given  $p,q,r$ satisfying (H1),(H2) and (H3) respectively and $\xi\in L^{q(\cdot)}(\O)$, we want to minimize the
functional
$$
I(v)=\int_{\O} \left(|\nabla
v(x)|^{p(x)}+|v(x)-\xi(x)|^{q(x)}\right)\, dx+\int_{\Gamma_N}
|v|^{r(x)}\, dS
$$
over all $v\in\A$, where
$$\A=\{v\in \wp: v-u_D \in \wpgama\},$$ $u_D\in \wp$ and $\partial\O=\Gamma_D\cup \Gamma_N$.
For the definitions of the variable exponent Sobolev spaces $\wp$
and $\wpgama,$ see Section \ref{appA1}.

\medskip

Let us observe that, considering  the applications we have in
mind, it is relevant to study the minimization problem in the case
where $p$ approaches the value $1$ in some regions. We can see, by
making some numerical experiments, that the minimizers have large
derivative in these regions. For this reason, the Conforming
Finite Element Method is not appropriate since, its use would
imply the need of fine meshes in order to obtain good
approximations, see Section \ref{examples}.

We  consider the so-called Discontinuous Galerkin Methods. These
methods are relatively new from the theoretical point of view. In
\cite{ABCM}, we can find a unification of all methods of this
type. In all the examples of that paper, the authors take as
model  a linear differential equation.

At this point we want to mention that, in \cite{BCE} and
\cite{CLR}  the authors find an approximation of the solutions by
using an explicit finite difference scheme for the associated
parabolic problem.

Our aim is to study, in the future, the minimization problem in
the case where $p$ approaches the value $1$ in some regions (where
there is no weak formulation). For this reason, we think that the
best way to find approximations is by finding a good
discretization of the minimization problem. We take a
discretization  similar to the one in \cite{BO} where the authors
study a functional that includes the case  $p=$ constant.

Our discrete functional is the following
\begin{align*}I_h(v_h)&=\int_{\O} \left(|\nabla v_h+R_h(v_h)|^{p(x)}+|v_h-\xi|^{q(x)}\right)\, dx
+\int_{\Gamma_D}|v_h-u_D|^{p(x)}{\bf{h}}^{1-p(x)}\,dS\\&
+\int_{\Gamma_{int}}|\[v_h\]|^{p(x)}{\bf{h}}^{1-p(x)}\, dS
+\int_{\Gamma_N} |v_h|^{r(x)}\, dS,
\end{align*}
where ${\bf{h}}$ is the local mesh size, $h$ is the global mesh
size, $\Gamma_{int}$ is the union of the interior edges of the
elements, $\[v_h\]$ is the jump of the function between two  edges
and $\nabla v_h$ denotes the elementwise gradient of $v_h,$ see
Section \ref{appA2} for a precise definition. Observe that the
boundary condition is weakly imposed by the second  term of the
functional. Lastly, $R_h$ is the lifting operator defined in
Section \ref{lifting}, which represents the contributions of the
jumps to the distributional gradient.

In the case $p=$ constant,  the boundedness of this operator is
proved by using an inf-sup condition. In our case, for technical
reasons, we use a different approach, which consists of finding a
local characterization of the operator.

With this setting the discrete problem is to find a minimizer
$u_h$ of $I_h$ over the space $S^k(\th)$ of all the functions that
are polynomials of degree at most $k$ in each element, with $k\ge
1,$ see Section \ref{appA2}.

In this work we show that the sequence $u_h$ converges to the
minimizer $u$ of $I$ over the space $\A$. We want to remark here
that we are assuming  that, for each $v_h\in  S^k(\th)$, all the
terms of $I_h(v_h)$ can be exactly computed.

In fact, we prove the following

\begin{teo}\label{conv}
 Let  $\O$ be a polyhedral domain.
 Let $p(x),q(x),$ and $r(x)$ be functions satisfying (H1), (H2) and (H3) respectively and let $u_D\in W^{2,p_2}(\O)$.
 For each $h\in (0,1]$, let $u_h\in S^k(\th)$ be the minimizer of ${I}_h$.
 If $u$ is the minimizer of ${I}$ then
\begin{align}
    \label{lr}u_{h}&\to u \mbox{ strongly in } L^{s(\cdot)}(\O) \quad \forall s \mbox{ satisfying (H2)}, \\
     \label{lr1}u_{h}&\to u \mbox{ strongly  in } L^{t(\cdot)}(\partial\O) \quad
     \forall t \mbox{ satisfying (H3)},\\
 \label{convIh} I_h(u_{h})&\rightarrow  I(u),\\
 \label{Ruh} R(u_h)&\to 0 \mbox{ strongly in } \lp,\\
 \label{convpena}
 \int_{\Gamma_D}|u_h-u_D|^{p(x)}{\bf{h}}^{1-p(x)}\,dS &
+\int_{\Gamma_{int}}|\[u_h\]|^{p(x)}{\bf{h}}^{1-p(x)}\, dS \rightarrow  0,\\
\label{fuerte}\nabla u_{h}&\to \nabla u \mbox{ strongly in } \lp.
 \end{align}
\end{teo}

\bigskip
Lastly, we want to mention the places where we  need  the
regularity  hypotheses on the function $p$. First, in order to
prove Theorem \ref{conv} we need to use the continuity of the
embedding $\wp\hookrightarrow\lpe,$ the continuity of the Trace
operator $\wp\hookrightarrow\lqb$ and the Poincar\'e inequality.
As we can see in Theorem \ref{embed}, Theorem  \ref{traza} and
Theorem \ref{poinc2} for these results we need $p$ to be
$\log$-H\"{o}lder and $r\in C^0({\partial\O})$.

We also use strongly that $p$ is $\log$-H\"{o}lder in Proposition
\ref{papa}. This result says that if $\kappa$ is an element with
diameter $h_{\kappa}$ and $p^{\kappa}_+$ and $p^{\kappa}_-$ are
respectively the maximum and minimum of $p$ over $\kappa$ then
$h_{\kappa}^{p^{\kappa}_--p^{\kappa}_+}$ is bounded independent of
$h_{\kappa}$. This property is crucial in the proof of several
results along the paper.

\medskip
 On the other hand, in order to prove the convergence of the sequence $u_h$
 we  need a technical hypothesis on the boundary condition $u_D$. In fact, Lemma \ref{auxiliar}  only covers the case where $u_D\in W^{2,p_2}(\O)$.

\subsection*{Outline of the paper}

\medskip

In Section \ref{appA1} we state several properties of the Variable
Exponent Sobolev Spaces.

In Section \ref{appA2} we give some definitions and properties
related to the mesh and to the Broken Sobolev Spaces.

In Section \ref{qh} we study the Reconstruction operator and we
prove some error estimates that are crucial for the rest of the
paper (Corollary \ref{globalthfin}).

In Section \ref{lifop} we prove the boundedness of the Lifting
operator (Theorem \ref{lifcont}).

In Section \ref{cm2} we prove the Broken Poincar\'{e} inequality
(Theorem \ref{brok-poncare}), the   coercivity of the functional
(Theorem \ref{cm}) and finally we give the proof of Theorem
\ref{conv}.

In Section \ref{confor} we study the convergence of the Conforming
Finite Element Method.

In Section \ref{examples} we give a 1d example  and compare both
conforming and non-conforming schemes.

\section{Preliminaries: The spaces $L^{p(\cdot)}(\Omega)$ and $W^{1,p(\cdot)}(\Omega)$}
\label{appA1}

We now introduce the spaces  $L^{p(\cdot)}(\Omega)$ and
$W^{1,p(\cdot)}(\Omega)$ and state some of their properties.

\medskip

Let $p \colon\Omega \to  [p_1,p_2]$ be a measurable bounded
function, called a variable exponent on $\Omega$ where
 $p_{1}:= ess inf \,p(x)$ and $p_{2} := ess sup \,p(x)$ with $1\le p_1\leq p_2<\infty$.

We define the variable exponent Lebesgue space
$L^{p(\cdot)}(\Omega)$ to consist of all measurable functions $u
\colon\Omega \to \R$ for which the modular
$$
\varrho_{p(\cdot)}(u) := \int_{\Omega} |u(x)|^{p(x)}\, dx
$$
is finite. We define the Luxemburg norm on this space by
$$
\|u\|_{L^{p(\cdot)}(\Omega)} = \|u\|_{p(\cdot)}  := \inf\{k >
0\colon \varrho_{p(\cdot)}(u/k)\leq 1 \}.
$$
This norm makes $L^{p(\cdot)}(\Omega)$ a Banach space.

\bigskip
The following properties can be obtained directly from the
definition of the norm,
\begin{prop}\label{equi}
If $u,u_n\in L^{p(\cdot)}(\Omega)$, $\|u\|_{p(\cdot)}=\lambda,$
then
\begin{enumerate}
\item $\lambda<1$ $(=1, >1)$ if only if
$\di\int_{\O}|u(x)|^{p(x)}\, dx<1\ (=1,>1);$ \item If $\lambda\geq
1,$ then $\di\lambda^{p_1}\le \int_{\O} |u(x)|^{p(x)}\, dx\leq
\lambda^{p_2};$ \item If $\lambda\leq 1,$ then
$\di\lambda^{p_2}\le \int_{\O} |u(x)|^{p(x)}\, dx\leq
\lambda^{p_1};$ \item $\di\int_{\O}|u_n(x)|^{p(x)}\, dx \to 0$ if
only if $\di\|u_n\|_{p(\cdot)}\to 0;$ \item
$\|1\|_{p(\cdot)}\le\max\left\{|\Omega|^{\frac{1}{p_1}},
     |\Omega|^{\frac{1}{p_2}}\right\};$
\item If $\di\Omega= \di\bigcup_{i=1}^m \O_i$ where $\O_i\subset
\O$ are open sets then there exists a constant $C>0$ depending on
$m$ such that
     $$\di\|u\|_{\lp}\leq C \sum_{i=1}^m \|u\|_{L^{p(\cdot)}(\O_i)}.$$
\end{enumerate}
\end{prop}
\begin{proof}
See Theorem 1.3 and Theorem 1.4  in \cite{FS}.
\end{proof}

For the proofs of the following three theorems we refer the reader
to \cite{KR}.

\begin{teo}\label{imb}
Let $q(x)\leq p(x)$, then
 $L^{p(\cdot)}(\Omega)\hookrightarrow L^{q(\cdot)}(\Omega)$
continuously.
\end{teo}

\begin{teo}\label{interp}
Let $p,q,r:\O\to [1,\infty)$ and $\ep>0$ be such that  $p(x)\le
r(x)<q(x)-\ep$ for all $x\in\Omega.$ Then, there exists a positive
constant $C$ such that for every $u\in L^{p(\cdot)}(\Omega)\cap
L^{q(\cdot)}(\Omega)$ the inequality
$$
\di\|u\|_{L^{r(\cdot)}(\O)}\leq C
\|u\|_{\lp}^{\mu}\|u\|_{\lq}^{\nu}
$$
holds, where $\mu>0$ and $\nu\ge0$ are define as
$$
\mu =
\begin{cases}
    {\di{\esssup_\O}}\di\frac{p(x)}{r(x)}\frac{q(x)-r(x)}{q(x)-p(x)} & \mbox{if } \|u\|_{\lp}>1,\\[.5cm]
    \di{\essinf_\O}\frac{p(x)}{r(x)}\frac{q(x)-r(x)}{q(x)-p(x)} & \mbox{if } \|u\|_{\lp}\le 1,
\end{cases}
$$
$$
\nu =
\begin{cases}
    \di{ \esssup_\O}\frac{q(x)}{r(x)}\frac{r(x)-p(x)}{q(x)-p(x)} & \mbox{if } \|u\|_{\lq}>1,\\[.5cm]
    \di{ \essinf_\O}\frac{q(x)}{r(x)}\frac{r(x)-p(x)}{q(x)-p(x)} & \mbox{if } \|u\|_{\lq}\le 1.\\
\end{cases}
$$

\end{teo}

\begin{teo}\label{ref}
Let $\di p'(x)$ such that, ${1}/{p(x)}+{1}/{p'(x)}=1.$ Then
$L^{p'(\cdot)}(\Omega)$ is the dual of $L^{p(\cdot)}(\Omega)$.
Moreover, if $p_1>1$, $L^{p(\cdot)}(\Omega)$ and
$W^{1,p(\cdot)}(\Omega)$ are reflexive.
\end{teo}

\medskip

Now we give  some well known inequalities,

\begin{prop}\label{desip} For any $x$ fixed we have the following inequalities
\begin{align*}
&|\eta-\xi|^{p(x)}\leq C (|\eta|^{p(x)-2} \eta-|\xi|^{p(x)-2} \xi)
(\eta-\xi)&\quad  \mbox{ if } p(x)\geq 2,\\
&
 |\eta-\xi|^2\Big(|\eta|+|\xi|\Big)^{p(x)-2}
\leq C (|\eta|^{p(x)-2} \eta-|\xi|^{p(x)-2} \xi)
(\eta-\xi)&\quad  \mbox{ if } p(x)< 2,\\
&|\eta|^{p(x)}\leq
2^{p(x)-1}(|\eta-\xi|^{p(x)}+|\xi|^{p(x)})&\quad  \mbox{ if }
p(x)\geq 1.
\end{align*}
\end{prop}

These inequalities say that the function $A(x,q)=|q|^{p(x)-2}q$ is
strictly monotone.

\begin{prop}\label{debil}
Let $F_n,F\in\lp.$
\begin{enumerate}
\item If
$$F_n\rightharpoonup F \mbox{ weakly in } \lp$$
then
$$\int_{\O}|F|^{p(x)}\, dx \leq \liminf_{n\to\infty}
\int_{\O}|F_n|^{p(x)}\, dx.$$

\item If
 $$ F_n \rightarrow F \mbox{ strongly in } \lp$$ then
$$\int_\O |F_n|^{p(x)}\, dx \to \int_\O |F|^{p(x)}\, dx.$$

\item If
 \begin{equation}\label{debilmasnorma}F_n\rightharpoonup F \mbox{ weakly in } \lp \quad \mbox{and}\quad\int_\O |F_n|^{p(x)}\, dx \to \int_\O |F|^{p(x)}\, dx\end{equation} then
 $$ F_n \rightarrow F \mbox{ strongly in } \lp.$$

\end{enumerate}

\end{prop}
\begin{proof}
For the proof of (1) and (3) see  Theorem 3.9 and Lemma 2.4.17  in
\cite{DHHR}. Finally (2) follows by Proposition 2.3 in \cite{FZ}.
%
\end{proof}

Let $W^{1,p(\cdot)}(\Omega)$ denote the space of measurable
functions $u$ such that, $u $ and the distributional derivative
$\nabla u$ are in $L^{p(\cdot)}(\Omega)$. The norm
$$
\|u\|_{W^{1,p(\cdot)}(\O)}\colon= \|u\|_{p(\cdot)} + \| |\nabla u|
\|_{p(\cdot)}
$$
makes $W^{1,p(\cdot)}(\Omega)$ a Banach space.

We define the space $W_0^{1,p(\cdot)}(\Omega)$ as the closure of
$C_0^{\infty}(\Omega)$ in $W^{1,p(\cdot)}(\Omega)$. Then we have
the following version of Poincar\'e inequality, see Theorem 3.10
in \cite{KR}.

\begin{lema}\label{poinc} If $p:\O\to[1,+\infty)$ is continuous in $\overline{\Omega}$,
    there exists a constant $C$ such that for every $u\in W_0^{1,p(\cdot)}(\Omega)$,
$$
\|u\|_{L^{p(\cdot)}(\Omega)}\leq C\|\nabla
u\|_{L^{p(\cdot)}(\Omega)}.
$$
\end{lema}

We also have the following version of the Poincar\'e inequality,
see Lemma 2.1 in \cite{HH},
\begin{teo}\label{poinc2} Let $\O\subset \R^n$ be a Lipschitz domain and $p,q:\O\to[1,+\infty)$ with   $p\leq q \leq p^*$. Then,
$$\|u-(u)_{\O}\|_{L^{q(\cdot)}(\Omega)}\leq C\|\nabla
u\|_{L^{p(\cdot)}(\Omega)}$$ for all $u\in \wp$, where
$(u)_{\O}=\frac1{|\O|}\int_{\O} u \, dx.$

\end{teo}

In order to have better properties of these spaces, we need more
hypotheses on the regularity of $p$. For expample, it was proved
in  \cite{D}, Theorem 3.7, that if one assumes that $\partial\O$
is Lipschitz and   $p$ is log-H\"{o}lder continuous then
$C^{\infty}(\bar{\Omega})$ is dense in $W^{1,p(\cdot)}(\Omega),$
see also  \cite{Di,DHN, Fanx2,KR,Sam1}. The local log-Hl\"{o}der
condition was first used in the variable exponent context in
\cite{Z}.

\medskip

We now state two Sobolev embedding Theorems  (for the proofs see
\cite{D3} and Corollary 2.4 in \cite{Fanx}, respectively).

\begin{teo}\label{embed}
Let $\Omega$ be a Lipschitz domain. Let $p:\Omega\to [1,\infty)$
be log-H\"{o}lder continuous. Then the embedding
$\wp\hookrightarrow \lpe$ is continuous.

\end{teo}

\begin{teo}\label{traza}
Let $\Omega$ be an open bounded   domain with Lipschitz boundary.
Suppose that $p\in C^0(\overline{\O})$ with $p_1>1$. If $r\in
C^0(\partial\O)$ satisfies the condition $ 1\leq r(x)<p_* (x)$ for
all $x\in\partial\O,$ then there is a compact boundary trace
embedding $\wp \hookrightarrow L^{r(\cdot)}(\partial\O)$.
\end{teo}

Let $\Gamma_D\subset \partial\O$, and $p$ be log-H\"{o}lder. We
define the space $\wpgama$ as the closure of the space $\{\varphi
\in C^{\infty}(\overline{\Omega}): \varphi=0 \mbox{ on }
\Gamma_D\}$ in $\wp$.

\medskip
The following proposition is crucial in order to prove the main
result of this paper.
\begin{prop}\label{papa}
Let $p:\O\to[1,\infty)$ be $\log$-H\"{o}lder continuous and
bounded. Let $\alpha>0,$ $D\subset \O$ and $h=\mbox{diam}(D)$ then:
\begin{enumerate}
\item There exists a constants $C$ independent of $h$ such that
\begin{equation}\label{epapa}
    h^{\alpha(p(x)-p(y))}\le C \quad \forall x,y\in \overline{D};
\end{equation}
\item If $A\ge h^\alpha$ then $A^{p(x)}\le C A^{p(y)}$ for all
$x,y\in \overline{D}$ such that $p(x)\le p(y)$.
\end{enumerate}
\end{prop}
\begin{proof}
    Let $x, y\in \overline{D}$. If $p(x)\geq p(y)$ or $h\ge 1$ the result follows since $\O$ is bounded. If $p(x)\leq p(y)$ and $h<1$, using that $p$ is $\log$-H\"{o}lder, we have
$$
p(y)-p(x)\leq \frac{C}{\log\left(e+\di\frac{1}{|x-y|}\right)}\leq
\frac{C}{\log\left(e+\di\frac{1}{h}\right)}.
$$
Then, we get \eqref{epapa}.

By \eqref{epapa} and as $A\ge h^\alpha,$ we have that for all
$x,y\in\overline{D}$ such that $p(x)\le p(y),$
$$
A^{p(x)}=A^{p(y)}\left(\frac{A}{h^\alpha}\right)^{p(x)-p(y)}h^{\alpha(p(x)-p(y))}\le
CA^{p(y)}.
$$
\end{proof}

\section{The mesh $\th$ and properties of $\wpt$}
\label{appA2}

In this section we give some definitions and properties related to
the mesh and to the Broken Sobolev Space.

\begin{hyp}\label{omega}
Let $\O$ be a polygonal Lipschitz domain and $(\th)_{h\in(0,1]}$
be a family of partitions of $\overline{\O}$ into polyhedral
elements. We assume that there exist a finite number of reference
polyhedral $\hat{\kappa}_1,...,\hat{\kappa}_r$ such that for all
$\kappa\in \th$ there exists an invertible affine map $F_{\kappa}$
such that, $\kappa=F_{\kappa}(\hat{\kappa}_i)$. We assume that
each $\kappa\in\th$ is closed and that $diam(\kappa)\leq h$ for
all $\kappa\in \th$.
\end{hyp}

Now we give some notation,
$$
\begin{aligned}\mathcal{E}_h&:=\{\kappa\cap\kappa': \dim_H(\kappa\cap\kappa')=N-1\}
\cup\{\kappa\cap\partial\O: \dim_H(\kappa\cap\partial\O)=N-1\},\\
\Gamma_{int}&:=\bigcup\{ e\in\eh: \dim_H(e\cap
\partial\O)<N-1\}.\end{aligned}$$

\medskip

$\nh$ is the set of nodes of $\th$. For every $z\in\nh$ and
$e\in\mathcal{E}_h$ we define,
$$\begin{aligned}
\di{T}_z&:=\bigcup\{{\kappa}\in\th\colon z\in{\kappa}\},\quad
{T}_{\kappa}:=\bigcup\{{{T}_z}\colon z\in{\kappa}\},\quad
{T}_{e}:=\bigcup\{{{T}_{\kappa}}\colon e\in{\kappa}\},
\\
h_{\kappa}&:=\mbox{diam}(\kappa),\quad \quad
h_{z}:=\mbox{diam}(T_z)
\quad \quad h_{e}:=\mbox{diam}(e),\\
p^{\kappa}_-&:=\ess\inf_{x\in\kappa} p(x)\quad\quad
p^{\kappa}_+:=\ess\sup_{x\in\kappa} p(x),\quad
p^{e}_-:=\ess\inf_{x\in e} p(x) \mbox{ and }
p^{e}_+:=\ess\sup_{x\in e} p(x).
\end{aligned}$$

\medskip

We assume that the mesh satisfies the following hypotheses,

\begin{hyp}\label{mesh}  The family of partitions $(\th)_{h\in(0,1]}$
    satisfies the Hypothesis \ref{omega} and
\begin{enumerate}
\item[(a)] There  exist positive constants  $C_1$ and $C_2$,
independent of $h$, such that for each element $\kappa\in\th$
$$
C_1h_{\kappa}^N\le|\kappa|\le C_2 h_\kappa^N.
$$
\item[(b)]There exists a constant $C_1>0$ such that for all $h\in
(0,1]$ and for all face
    $e\in\eh$ there exists a point $x_e\in e$ and a radius
    $\rho_e\geq C_1 \mbox{diam}(e)$ such that $B_{\rho_e}(x_e)\cap A_e\subset e$, where
    $A_e$ is the affine hyperplane spanned by $e$.
    Moreover, there are positive constants such that
        $$ch_{\kappa}\leq h_e \leq C h_{\kappa}, \quad c h_{\kappa'}\leq h_e\leq C h_{\kappa'}$$
        where $e=\kappa\cap\kappa'$.
\end{enumerate}

\end{hyp}

We use the notation $\sim$ to compare quantities which differ only
up to positive constants that do not depend on the local or global
mesh size or on any function which appears in the estimate.

\begin{remark}\label{cardinal}
By the regularity assumption of the mesh we have the following,
$$
\sharp\{z\in\mathcal{N}_h\colon z\in\kappa\}\sim 1, \quad
\quad \sharp \{\kappa\in\mathcal{T}_h\colon\kappa\subset T_z\}
\sim 1,$$
$$\sharp \{\kappa'\in\mathcal{T}_h\colon \kappa'\subset T_{\kappa}\} \sim 1, \quad
\sharp \{e\in\mathcal{E}_h\colon e\subset T_z\} \sim 1 \quad
\mbox{and}\quad \sharp \{ e\in\mathcal{E}_h\colon e\subset
T_{\kappa}\} \sim 1.$$
\end{remark}

\begin{remark}\label{comp}
 As a consequence,
 we have  that
 $\mbox{diam}(T_{\kappa})\sim h_{\kappa}$ and for each $z\in\kappa$ and $e\subset \partial \kappa$,
 $h_z\sim h_{\kappa}$ and $h_e \sim h_{\kappa}$. See the discussion on Section 4.2 in \cite{BO}.
\end{remark}

\begin{remark}\label{comp2}
By Proposition \ref{papa},  we also have that for each edge
 $e\subset \partial \kappa$,
 $h_{\kappa}^{p(x)}\sim h_{e}^{p(y)}$ for any $x,y\in \kappa$. We will  replace $p_-^{\kappa},p^{e}_-$ by $p_-$
 and $p^{\kappa}_+,p^{e}_+$ by $p_+$ when no confusion can arise.
\end{remark}
\medskip

Now, we introduce  the finite element spaces associated with
$\th.$ We define the variable broken Sobolev space as
$$
W^{1,p(\cdot)}(\th):=\{u\in L^1(\Omega)\colon u|_{\kappa}\in
W^{1,p(\cdot)}(\kappa) \mbox{ for all } \kappa \in\th\},
$$
and the subspaces
$$
U^k(\th):=\{u\in C(\O)\colon u|_{\kappa}\in P^k \mbox{ for all }
\kappa \in\th\},
$$

$$
S^k(\th):=\{u\in L^1(\Omega)\colon u|_{\kappa}\in P^k \mbox{ for
all } \kappa \in\th\}
$$
where $P^k$ is the space of polynomials functions of degree at
most $k\ge1.$

\medskip
We also define, for any $\kappa\in\th,$   the space
\begin{align*}
W^{1,p(\cdot)}(T_\kappa)&:=\{u|_{T_\kappa} \colon  u \in
W^{1,p(\cdot)}(\th)\},\
\end{align*}
and in the same manner we define the spaces $W^{1,p(\cdot)}(T_z)$
and  $W^{1,p(\cdot)}(T_e)$, for any
 $z\in{\N}_h $ and $e\in \mathcal{E}_h$ .

\medskip
For each face $e \in\mathcal{E}_h,$  $e\subset\Gamma_{int}$ we
denote by $\kappa^+$ and $\kappa^-$ its neighboring elements. We
write $\nu^+,\nu^-$ to denote the outward normal unit vectors to
the boundaries $\partial\kappa^{\pm}$, respectively. The jump of a
function $u\in\wpt$ and the average of a vector-valued function
$\phi\in(\wpt)^N,$ with traces $u^{\pm},$ $\phi^{\pm}$ from
$k^{\pm}$ are, respectively, defined as the vectors
$$
    \[u\]:=u^+ \nu^++u^- \nu^- \quad\mbox{ and }\quad
    \{\phi\}:= \frac{\phi^++ \phi^-}{2}.
$$


Let ${\bf h}:\partial\Omega\cup \Gamma_{int}\to\R$ a piecewise
constant function define by
$$
{\bf h}(x)=
    diam(e)  \mbox{ if }  x\in e,
$$
 where  $e\in \mathcal{E}_h$.

We consider the following seminorms on $W^{1,p(\cdot)}(\th)$,
$$
|u|_{W^{1,p(\cdot)}(\th)}=\|\nabla u\|_{\lp}+ \|\[u\]
{\bf{h}}^{\frac{-1}{p'(x)}}\|_{L^{p(\cdot)}(\Gamma_{int})},
$$
$$
|u|_{W_D^{1,p(\cdot)}(\th)}=|u|_{W^{1,p(\cdot)}(\th)}+ \| u
{\bf{h}}^{\frac{-1}{p'(x)}}\|_{L^{p(\cdot)}(\Gamma_{D})},
$$
and the following local seminorm
\begin{align*}
|u|_{W^{1,p(\cdot)}(T_{\kappa})}&=\|\nabla
u\|_{L^{p(\cdot)}(T_{\kappa})}+ \sum_{e\subset T_{\kappa}}\|
\[u\]{\bf{h}}^{\frac{-1}{p'(x)}}\|_{L^{p(\cdot)}(e)},
\end{align*}
for  any $\kappa\in\th$. Similarly, we define  the seminorms
$|u|_{W^{1,p(\cdot)}(T_{z})}$ and $|u|_{W^{1,p(\cdot)}(T_{e})}$
for  any
 $z\in{\N}_h $ and $e\in \mathcal{E}_h$ .

\begin{lema}\label{cotaD}
For all $p\colon[1,\infty)\to \R$, there exist a constant $C$,
independent of $h$ such that,
$$ |Du|(\O)\leq C|u|_{\wpt}\quad \forall u\in\wpt,\ \forall h\in (0,1].$$
\end{lema}
\begin{proof}
For all $u\in\wpt,$ we have that
$$
|Du|(\O)\le\int_\O |\nabla u| \, dx + \int_{\Gamma_{int}} |\[u\]|
\, ds.
$$
Thus, by H\"older inequality, Proposition \ref{equi} (5)  and the
Hypothesis \ref{mesh}, there exists a constant $C$ depending only
of $|\O|,$ $p_1$ and $p_2$ such that
$$
|Du|(\O)\le C\left(\|\nabla u\|_{\lp}+
\|{\bf{h}}^{\frac{-1}{p'(x)}}\[u\]\|_{L^{p(\cdot)}(\Gamma_{int})}\right).
$$
The proof is now complete.
\end{proof}

\begin{lema}\label{lema21}
Let $(\th)_{h\in(0,1]}$ be a family of partitions of $\O.$ Then,
for each function \mbox{$p,q\colon\Omega\to [1,\infty)$}, there
exists a constant $C>0$ independent of $h$, such that for any
$\kappa\in\th$
$$\|u\|_{L^{p(\cdot)}(\kappa)}\leq C h_{\kappa}^{\frac{N}{p_+}-\frac{N}{q_-}}
\|u\|_{L^{q(\cdot)}(\kappa)} \quad \forall u\in S^k(\th), \;
\forall h\in(0,1].$$
\end{lema}
\begin{proof}
Let $\kappa\in\th$, $\hat{\kappa}$ its corresponding reference
element and $F_{\kappa} \colon\hat{\kappa}\to {\kappa}$ the
associated affine mapping. We set $J=|\det(DF_{\kappa})|$. Using
Hypothesis \ref{mesh}, we have $C^{-1}h_{\kappa}^N\leq J\leq C
h_{\kappa}^N$, for some constant $C$ which is independent of
$\kappa$. Let $K>0$, then we have
$$\int_{\kappa} \left(\frac{|u|}{K}\right)^{p(x)}\, dx
=\int_{\hat{\kappa}} \left(\frac{|u\circ
F_{\kappa}|}{K}\right)^{p\circ F_{\kappa}(x)}J\, dx\leq C
h_{\kappa}^N\int_{\hat\kappa}\left(\frac{|u\circ
F_{\kappa}|}{K}\right)^{p\circ F_{\kappa}(x)}\, dx. $$ Thus,
$$\| (Ch_{\kappa}^N)^{-1/{p(x)}}u\|_{L^{p(\cdot)}(\kappa) }\le
\|u\circ F_{\kappa}\|_{L^{p\circ
F_{\kappa}(\cdot)}(\hat{\kappa})}.$$ Using that $h_{\kappa}\ll 1$,
we obtain

\begin{equation}\label{uf1}\|u\|_{L^{p(\cdot)}(\kappa) }\le
(Ch_{\kappa}^N)^{1/p_+} \|u\circ F_{\kappa}\|_{L^{p\circ
F_{\kappa}(\cdot)}(\hat{\kappa})}.
\end{equation}

Similarly, we have
\begin{equation}\label{uf2}
\|u\circ F_\kappa\|_{L^{q\circ F_\kappa(\cdot)}(\hat{\kappa}) }\le
(Ch_{\kappa}^{-N})^{1/q_-}
\|u\|_{L^{q(\cdot)}({\kappa})}.\end{equation}

As on a finite dimensional space all the norms are equivalent, we
have that there exists a constant $\bar{C}$ depending only on $N$
and $k$ such that,
\begin{equation}\label{eqnorm}
\|u\circ F_{\kappa}\|_{L^{p\circ
F_{\kappa}(\cdot)}(\hat{\kappa})}\leq C \|u\circ
F_{\kappa}\|_{L^{p_2}(\hat{\kappa})}\leq C \|u\circ
F_{\kappa}\|_{L^{q_1}(\hat{\kappa})}\leq \bar{C} \|u\circ
F\|_{L^{q(\cdot)}(\hat{\kappa})},\end{equation} where in the first
and last inequalities we are using Theorem \ref{imb}.

Finally, by \eqref{uf1}--\eqref{eqnorm} we arrive at the desired
result.
\end{proof}

\begin{lema}\label{local}
If $p$ is $\log$-H\"{o}lder continuous then, for any $e\in
\mathcal{E}_h\cap\partial\Omega$ and $z\in \mathcal{N}_h\cap e$
 we have that,
\begin{equation}\label{facil}
\|u\|_{L^{^p(\cdot)}(e)} \leq C  h_{z}^{-\frac{1}{p_{-}}}
\|u\|_{L^{p(\cdot)}(T_z)} \quad \forall u\in S^k(\th),
\end{equation}
where $C=C(p_1,p_2,N,\Omega,C_{log})$.

\end{lema}

\begin{proof}
Let $\kappa\in\mathcal{T}_h$ such that $e\subset \kappa$. Let
$F_{\kappa}$ and $\hat{\kappa}$ be as in the proof of Lemma
\ref{lema21} and  let  $\hat{e}=F^{-1}_{\kappa}(e)$.

Then,
$$\int_{e} \left(\frac{|u(x)|}{k}\right)^{p(x)}\, dS\leq C h_{\kappa}^{N-1}\int_{
\hat{e}} \left(\frac{|u\circ F_{\kappa}(x)|}{k}\right)^{p\circ
F_{\kappa}(x)}\, dS.$$ Hence,
$$\Big\|(C^{-1}h_{\kappa})^{\frac{1}{p(x)}}\frac{u}{h_{\kappa}^{N/p(x)}}\Big\|_
{L^{p(\cdot)}(e)}\leq \|u\circ F_{\kappa}\|_{L^{p\circ
F_{\kappa}(\cdot)}(\hat{e})}.$$

By using Theorem \ref{imb} and that all the norms are equivalent,
we have
$$
\|u\circ F_{\kappa}\|_{L^{p\circ F_{\kappa}(\cdot)}(\hat{e})} \leq
C \|u\circ F_{\kappa}\|_{L^{p_2}(\hat{e})} \leq C \|u\circ
F_{\kappa}\|_{L^{1}(\hat{e})}.
$$
On the other hand, by the local inverse estimate in \cite[page 837
]{BO}, we have
$$
\|u\circ F_{\kappa}\|_{L^{1}(\hat{e})}\leq C \|u\circ
F_{\kappa}\|_{L^{1}(\hat{\kappa})}.$$ By using again Theorem
\ref{imb}, we obtain
$$
\|u\circ F_{\kappa}\|_{L^{1}(\hat{\kappa})} \leq C \|u\circ
F_{\kappa}\|_{L^{p\circ F_{\kappa}(\cdot)}(\hat{\kappa})}.
$$
By using all the inequalities and the definition of the Luxembourg
norm, we arrive at

$$\Big\|{h_{\kappa}}^{\frac{1}{p(x)}}\frac{u}{h_{\kappa}^{N/p(x)}}\Big\|_{L^{p(\cdot)}(e)}\leq C
\Big\|\frac{u}{h_{\kappa}^{N/p(x)}}\Big\|_{L^{p(\cdot)}({\kappa})}.$$
Finally, we obtain
$$
\|{h_{\kappa}}^{\frac{1}{p(x)}}u\|_ {L^{p(\cdot)}(e)}\leq C
h_{\kappa}^{\frac{N(p_--p_+)}{p_-p_+}}
\|u\|_{L^{p(\cdot)}({\kappa})},
$$

By Remark \ref{papa}, we get
$$
\|{h_{\kappa}}^{\frac{1}{p(x)}}u\|_ {L^{p(\cdot)}(e)}\leq C
e^{N\frac{C}{p_1^2}} \|u\|_{L^{p(\cdot)}({\kappa})}.
$$

Now, inequality  \eqref{facil} follows immediately using
Proposition \ref{equi}(3) and the fact that $h_z \sim h_{\kappa}$.
\end{proof}
\medskip

The next result establishes the existence of the local projector
operator (for the proof see Subsection 3.1 of \cite{BO}).

\begin{lema}\label{lema6}
For all $z\in{\N}_{h}$ there exists a linear map
$\pi_z:BV(\O)\to\R$ such that
$$
\|u-\pi_z(u)\|_{L^1({T}_z)}\le C h_z|Du|({T}_z) \quad \forall u\in
BV(\O)
$$
where $C$ is a constant independent of $h$ and $z.$
\end{lema}

\section{The reconstruction operator $Q_h$}\label{qh}
In many   Discontinuous Galerkin problems one uses a priori bounds
in order to prove the Poincar\'{e} inequality for the discrete
space. In order to prove these inequalities   it is required to
use a reconstruction operator. In this section we define, as in
\cite{BO}, a family of quasi-interpolant operators and prove some
error estimates depending on the mesh size. These results are
more general than the ones in \cite{BO}, because we  prove bounds
in the variable $p(x)-$ norm. On the other hand these results are
less general than the ones in \cite{BR} in the sense that they
only cover the case of the finite dimensional space
$\mathcal{S}^k(\th)$. This last restriction comes from the fact
that in Lemma \ref{lema21} we need to use the equivalence of the
norms in the space of polynomials.

In order to prove these error estimates we  strongly use
Proposition \ref{papa}.  This is the reason why we need $p$ to be
$\log$-H\"{o}lder continuous.

\medskip
Now, we define and study the reconstruction operator. For each
$h\in(0,1],$ let
$$Q_h\colon S^k(\th)\to W^{1,\infty}(\O)$$ be the linear
operator defined by
$$
Q_h(u)=\sum_{z\in{\N}_h}\pi_z(u)\lambda_z,
$$
where $\lambda_z$ is the standard $P^1$ nodal basis function
associated with the vertex $z$ on the mesh ${\th}$.

\medskip

In the next theorem, we give some local estimates of the
$L^{q(\cdot)}(\kappa)$ and $L^{q(\cdot)}(e)$  norms in terms of
the   ${W^{1,p(\cdot)}(T_{\kappa})}$ seminorm and  $h$.

\begin{teo}\label{tqh}
Let $p, q\colon\O\to [1,\infty)$ be $\log$-H\"{o}lder continuous
in $\overline{\O}$. Then, the operator $Q_h$ satisfies
\begin{align}
\label{tqh1}\|u-Q_h(u)\|_{L^{q(\cdot)}(\kappa)}\le& C
h_{\kappa}^{\frac{N}{q_-}-\frac{N}{p_-} +1
}|u|_{W^{1,p(\cdot)}(T_{\kappa})} \quad\quad\,\,\, \forall \kappa \in \th,\\
\label{tqh2}\|u-Q_h(u)\|_{L^{q(\cdot)}(e)}\le&
Ch_e^{\frac{N-1}{q_-}-\frac{N}{p_-}+1} |u|_{W^{1,p(\cdot)}(
T_{e})}
 \quad \forall e \in \mathcal{E}_h\cap\partial\O,\\
\label{tqh3}\|\nabla Q_h(u)\|_{L^{p(\cdot)}(\kappa)}\le&C
|u|_{W^{1,p(\cdot)}(T_{\kappa})} \qquad \qquad\quad\qquad \forall
\kappa \in \th,
\end{align}
for all $u\in  S^k(\th)$ where  $C$ is a constant independent of
$h.$
\end{teo}

\begin{proof} We proceed in three steps.

{\it Step 1.} We first show inequality \eqref{tqh1}.

Fix $\kappa\in \mathcal{T}_h.$ For $z\in{\N}_h\cap\kappa,$ by
using  Proposition \ref{equi} (6) and Lemma \ref{lema21}, we get
\begin{align*}
\|u-\pi_z(u)\|_{L^{q(\cdot)}({T}_z)}&\le C
\sum_{\{\kappa':\kappa'\subset T_z\}}
h_{\kappa'}^{\frac{N}{q_+}-N}\|u-\pi_z(u)\|_{L^1(\kappa')}.
\end{align*}
By Remark \ref{comp} and Proposition \ref{papa}, we get
\begin{align*}
\|u-\pi_z(u)\|_{L^{q(\cdot)}({T}_z)}&\le C
h_z^{\frac{N}{q_-}-N}\|u-\pi_z(u)\|_{L^1({T}_z)}.
\end{align*}

Thus, by Lemma \ref{lema6}, we have
\begin{align*}
\|u-\pi_z(u)\|_{L^{q(\cdot)}({T}_z)}
&\le C  h_z^{\frac{N}{q_-}-N+1}\left(\|\nabla u\|_{L^1(T_z)}+
\sum_{e\subset T_z}\int_e |\[u\]|\,ds\right).
\end{align*}

Then, by using again   Lemma \ref{lema21} and Remark
\ref{cardinal} we have
\begin{equation}\label{ec}
\|u-\pi_z(u)\|_{L^{q(\cdot)}({T}_z)}\le
Ch_z^{\frac{N}{q_-}+1}\left( h_z^{-\frac{N}{p_-}}\|\nabla
u\|_{L^{p(\cdot)}(T_z)} + h_z^{-N}\sum_{e\subset T_z}\int_e
|\[u\]|\,ds\right).
\end{equation}
To estimate the second term, we use H\"{o}lder inequality,
obtaining
\begin{equation}\label{holderene}
\begin{aligned}
\int_e |\[u\]|\, ds\le&2\|\ \[u\]h_e^{-\frac{1}{p'(x)}}\|_{L^{p(\cdot)}(e)}\|h_e^{\frac{1}{p'(x)}}\|_{L^{p'(\cdot)}(e)}\\
\le& C\|\
\[u\]h_e^{-\frac{1}{p'(x)}}\|_{L^{p(\cdot)}(e)}h_e^{1-\frac{1}{p_-}}\|1\|_{L^{p'(\cdot)}(e)}.
\end{aligned}\end{equation}

Now, by Proposition \ref{equi} (5), we have that
$$
\|1\|_{L^{p'(\cdot)}(e)}\le C h_e^{(N-1)(1-\frac{1}{p_-})}.
$$
Then, we obtain
$$
\int_e |\[u\]|\, ds\le
C\|\[u\]h_e^{-\frac{1}{p'(x)}}\|_{L^{p(\cdot)}(e)}h_z^{N(1-\frac{1}{p_-})}.
$$
Therefore, summing over all $e\subset T_z$ and using \eqref{ec},
we arrive at
\begin{equation}\label{despi}
\|u-\pi_z(u)\|_{L^{q(\cdot)}({T}_z)}\le C
h_z^{\frac{N}{q_-}-\frac{N}{p_-}+1}|u|_{W^{1,p(\cdot)}( T_{z})}.
\end{equation}

Now, as in the proof Theorem 3.1 in \cite{BO} and using
Proposition \ref{equi} (6), we have the inequality \eqref{tqh1}.

\medskip

{\it Step 2.} We now show the inequality \eqref{tqh2}.

Fix $e\in\mathcal{E}_h\cap\partial\O$ and let $z\in{\N}_h\cap e.$
By the inequality \eqref{facil},
$$
\|u-\pi_z(u)\|_{L^{q(\cdot)}(e)}\le C
h_{z}^{-\frac{1}{q_-}}\|u-\pi_z(u)\|_{L^{q(\cdot)}({T}_z)}.
$$
Again, following the lines in \cite{BO} and using that $p$ and $q$
are $\log$-H\"{o}lder continuous in $\overline{\O},$ we arrive at
the inequality \eqref{tqh2}.

\medskip

{\it Step 3.} Finally, we will show the inequality \eqref{tqh3}.

Fix $\kappa\in \mathcal{T}_h.$ First, since $(\lambda_z)_{z\in
{\mathcal{N}}_h\cap\kappa}$ is a partition of the unity in
$\kappa$ we have that for any $x\in \kappa$
$$\nabla Q_h u(x)-\nabla u(x)=\sum_{z\in{\mathcal{N}}_h\cap\kappa} (\pi_z(u)-u(x))\nabla\lambda_z(x) + \sum_{z\in{\mathcal{N}}_h\cap\kappa}\nabla u(x)\lambda_z(x)$$
$$
\|\nabla Q_h u\|_{L^{p(\cdot)}(\kappa)}
\le\sum_{z\in{\mathcal{N}}_h\cap\kappa}
\|(\pi_z(u)-u)\nabla\lambda_z\|_{L^{p(\cdot)}(\kappa)}+\sum_{z\in{\mathcal{N}}_h\cap\kappa}
\|\nabla u\lambda_z\|_{L^{p(\cdot)}(\kappa)}+\|\nabla
u\|_{L^{p(\cdot)}(\kappa)}.
$$
Now, using  Hypothesis \eqref{mesh}, we have that there exists a
constant $C_1$ such that $|\nabla\lambda_z|<C_1h^{-1}$ in
$\kappa$, and by \eqref{despi} we get, using Remark
\ref{cardinal},
\begin{align*}
\|\nabla Q_h u\|_{L^{p(\cdot)}(\kappa)} &\le
C\sum_{z\in{\mathcal{N}}_h\cap\kappa}|u|_{W^{1,p(\cdot)}(T_{z})}+|u|_{W^{1,p(\cdot)}(T_{\kappa})}
\leq (C+1)|u|_{W^{1,p(\cdot)}(T_{\kappa})}.
\end{align*}

The proof is now complete.
\end{proof}

\medskip

Our next aim is to prove some global estimates. To this end we
will need some definitions.

\begin{defi}\label{p*}
Let $p:\O\to[1,\infty)$. Given $q:\O\to[1,\infty)$ and $q\le p^*$
in ${\O},$ we define
$$
\gamma=\sup\left\{\frac{q(x)}{p^*(x)}\colon x\in{\O}\right\}.
$$
\end{defi}

Observe that $0\le\gamma\le1$ and $\gamma=0$ if $p(x)\ge N$ for
all $x\in \O$ and $\gamma=1$ if $p(x)<N$ and $q(x)=p^*(x)$ for all
$x\in\O$.

\begin{defi}\label{p_*}
Let $p:\O\to[1,\infty)$. Given $q:\O\to[1,\infty)$ and $q\le p_*$
in ${\O},$ we define
$$
\beta=\sup\left\{\frac{q(x)}{p_*(x)}\colon x\in{\O}\right\}.
$$
\end{defi}

Observe that $0\le\beta\le1$ and $\beta=0$ if $p(x)\ge N$ for all
$x\in \O$ and $\beta=1$ if $p(x)<N$ and $q(x)=p_*(x)$ for all
$x\in\O$.

\begin{lema}\label{globalth}
    Let $p,q:\O\to[1,\infty)$ be $\log$-H\"{o}lder continuous in $\overline{\O}.$ Let $u\in  S^k(\th)$ satisfy
\begin{equation}\label{hipacot}
|u|_{W^{1,p(\cdot)}(\mathcal{T}_h)}\leq 1.
\end{equation}
Then, we have that
\begin{itemize}
\item If  $p\le q\le p^*$ in $\overline{\O},$ then
\begin{align}
\label{qh-global}\int_{\O}|u-Q_h(u)|^{q(x)}\, dx \le& C h^{N(1-\gamma)}, \\
\label{qh-grad}\int_\O |\nabla Q_h(u)|^{p(x)}\, dx\le C,
\end{align}
\item If  $p\le q\le p_*$ in $\overline{\O},$  then
\begin{equation}
\label{qh-traza} \int_{\partial\O}|u-Q_h(u)|^{q(x)}\, dS\le C
h^{(N-1)(1-\beta)},
\end{equation}
\end{itemize}
where  $C=C(p_1,p_2,\Omega,C_{log},N)$ and $\gamma$ and $\beta$
are given in Definitions \ref{p*} and \ref{p_*}, respectively.
\end{lema}

\begin{proof}
First observe that, by \eqref{tqh1}, we have
$$
\int_\kappa
\frac{|u-Q_h(u)|^{q(x)}}{\left(Ch_k^{\frac{N}{q_-}-\frac{N}{p_-}+1}
|u|_{W^{1,p(\cdot)}(T_{\kappa})}\right)^{q(x)}}dx \le1\quad
\forall \kappa\in\mathcal{T}_h,
$$
and by \eqref{hipacot}, we get
$$
\frac{1}{Ch_k^{N-\frac{Nq_-}{p_-}+q_-}|u|_{W^{1,p(\cdot)}(T_{\kappa})}^{q_-}}\int_\kappa
|u-Q_h(u)|^{q(x)}dx\le1\quad \forall \kappa\in\mathcal{T}_h.
$$
Then, by Proposition \ref{papa},
\begin{align*}
\int_{\kappa} |u-Q_h(u)|^{q(x)}\, dx&\le
Ch_{\kappa}^{N-\frac{Nq_-}{p_-}+q_-} |u|_{W^{1,p(\cdot)}(
T_{\kappa})}^{q_-}\le C h_{\kappa}^{N-\frac{Nq(x)}{p(x)}+q(x)}
|u|_{W^{1,p(\cdot)}(T_{\kappa})}^{q_-}
\end{align*}
for any $\kappa\in\mathcal{T}_h$ and  $x\in\kappa$. Therefore,
\begin{equation}\label{casi}
\int_{\kappa} |u-Q_h(u)|^{q(x)}\, dx\le  C h^{N(1-\gamma)}
|u|_{W^{1,p(\cdot)}(T_{\kappa})}^{q_-}\quad
\forall\kappa\in\mathcal{T}_h.
\end{equation}

On the other hand, by Remark \ref{cardinal}, the number of
$\kappa\subset T_{\kappa}$ is uniformly bounded in $h$. Using this
fact and  Proposition \ref{equi} (6), we have that
\begin{equation}\label{alfin}
|u|^{q_-}_{W^{1,p(\cdot)}( T_{\kappa})}\leq C \sum_{\kappa\subset
T_{\kappa}} \left(\|\nabla u\|^{q_-}_{\lpk}+\|\[u\]
{\bf{h}}^{\frac{1-p}{p}}\|^{q_-}_{L^{p(\cdot)}(\kappa\cap\Gamma_{int})}\right).\end{equation}

On the other hand, if we suppose that $\|\nabla u\|_{\lpk}\geq
h_{\kappa}^{N/q_-}$, by Proposition \ref{papa} (2), we have that
\begin{equation}\label{mayor1}
\|\nabla u\|_{\lpk}^{q_-}\le C \|\nabla u\|_{\lpk}^{q_+}.
\end{equation}

Arguing as before, if $\|\[u\]
{\bf{h}}^{\frac{1-p}{p}}\|_{L^{p(\cdot)}(\kappa\cap\Gamma_{int})}\geq
h_{\kappa}^{N/q_-},$ we have that
\begin{equation}\label{mayor12}
\|\[u\]
{\bf{h}}^{\frac{1-p}{p}}\|_{L^{p(\cdot)}(\kappa\cap\Gamma_{int})}^{q_-}\leq
C \|\[u\]
{\bf{h}}^{\frac{1-p}{p}}\|_{L^{p(\cdot)}(\kappa\cap\Gamma_{int})}^{q_+}.
\end{equation}

Now, we take
$$A=\left\{\kappa\in \mathcal{T}_h: \|\nabla u\|_{\lpk}\geq
h_{\kappa}^{N/q_-}\right\},$$ and
$$B=\left\{\kappa\in \mathcal{T}_h: \|\[u\]
{\bf{h}}^{\frac{1-p}{p}}\|_{L^{p(\cdot)}(\kappa\cap\Gamma_{int})}\geq
h_{\kappa}^{N/q_-}\right\}.$$

Observe that
\begin{equation}\label{truco1}
\begin{aligned}
\sum_{\kappa\in A^c} \|\nabla u\|_{\lpk}^{q_-}\leq
\sum_{\kappa\in A^c} h_{\kappa}^N\leq C,\\
\sum_{\kappa\in B^c} \|\[u\]
{\bf{h}}^{\frac{1-p}{p}}\|^{q_-}_{L^{p(\cdot)}(\kappa\cap\Gamma_{int})}\leq
C.\end{aligned}\end{equation}

On the other hand, by hypothesis \eqref{hipacot}, we have that
$\|\nabla u\|_{\lpk}\leq 1$ and then, for all $\kappa\in\th$
\begin{equation}\label{truco2}\begin{aligned}&
\|\nabla u\|_{\lpk}^{q_+}\le\|\nabla u\|_{\lpk}^{p_+}\leq
\int_{\kappa}|\nabla u|^{p(x)}\,
dx,\\
&\|\[u\]
{\bf{h}}^{\frac{1-p}{p}}\|^{q_+}_{L^{p(\cdot)}(\kappa\cap\Gamma_{int})}\leq
\int_{\kappa\cap\Gamma_{int}}|\[u\]|^{p(x)} {\bf{h}}^{1-p(x)}\,
dx.\end{aligned}\end{equation}

Since each $\kappa$ appears only in finitely many sets
$T_{\kappa'}$ we have, by  \eqref{alfin}--\eqref{truco2},
\begin{align*} \sum_{\kappa\in \mathcal{T}_h}|u|^{q_-}_{W^{1,p(\cdot)}(T_{\kappa})} \leq & C\left(
\sum_{\kappa\in A} \|\nabla u\|^{q_+}_{\lpk}+\sum_{\kappa\in A^c}
\|\nabla u\|^{q_-}_{\lpk}\right)\\&+C\left(\sum_{\kappa\in
B}\|\[u\]
{\bf{h}}^{\frac{1-p}{p}}\|^{q_+}_{L^{p(\cdot)}(\kappa\cap\Gamma_{int})}+\sum_{\kappa\in
B^c}\|\[u\]
{\bf{h}}^{\frac{1-p}{p}}\|^{q_-}_{L^{p(\cdot)}(\kappa\cap\Gamma_{int})}\right)\\
\leq & C \left(\sum_{\kappa\in A}\int_{\kappa}|\nabla u|^{p(x)}\,
dx +\sum_{\kappa\in B}\int_{\kappa\cap\Gamma_{int}}|\[u\]|^{p(x)}
{\bf{h}}^{1-p(x)}\, ds+1\right)\\=&  C\left( \int_{\O}|\nabla
u|^{p(x)}\, dx +\int_{\Gamma_{int}}|\[u\]|^{p(x)}
{\bf{h}}^{1-p(x)}\, ds+1\right).
\end{align*}

Thus, by \eqref{hipacot} and \eqref{casi} we get,
\begin{align*}
\int_{\O}|u-Q_h(u)|^{q(x)}\,
dx&=\sum_{\kappa\in\mathcal{T}_h}\int_{\kappa}|u-Q_h(u)|^{q(x)}\,
dx
\le C h^{N(1-\gamma)}.
\end{align*}

Lastly, using the same argument, \eqref{tqh2} and \eqref{tqh3}, we
get
$$
\int_{\partial\O}|u-Q_h(u)|^{q(x)}\, dS \leq C h^{(N-1)(1-\beta)}
\qquad \mbox{and}\qquad \int_{\O}|\nabla Q_h(u)|^{p(x)}\, dx \leq
C ,
$$
where $C$ is independent of $h$.
\end{proof}

The following corollary follows immediately
\begin{corol}\label{globalthfin}
    Let $p,q:\O\to[1,\infty)$ be log-H\"{o}lder continuous in $\overline{\O}$.
Then, for all $u\in  S^k(\th),$ we have,
\begin{itemize}
\item If  $p\le q\le p^*$ in ${\O},$ then
\begin{align}
\label{qh-global-fin}\int_{\O}|u-Q_h(u)|^{q(x)}\, dx \le& C
h^{N(1-\gamma)}
\max\left\{|u|_{{W}^{1,p(\cdot)}(\mathcal{T}_h)}^{q_1},|u|_{{W}^{1,p(\cdot)}(\mathcal{T}_h)}^{q_2}\right\}, \\
\label{qh-global-grad}\int_\O |\nabla Q_h(u)|^{p(x)}\, dx\le&
C\max\left\{|u|_{{W}^{1,p(\cdot)}(\mathcal{T}_h)}^{p_1},|u|_{{W}^{1,p(\cdot)}(\mathcal{T}_h)}^{p_2}\right\}.
\end{align}
\item If  $p\le q\le p_*$ in ${\O},$  then
\begin{equation}
\label{qh-traza2} \int_{\partial\O}|u-Q_h(u)|^{q(x)}\, dS\le C
h^{(N-1)(1-\beta)}
\max\left\{|u|_{{W}^{1,p(\cdot)}(\mathcal{T}_h)}^{q_1},|u|_{{W}^{1,p(\cdot)}(\mathcal{T}_h)}^{q_2}\right\}.
\end{equation}
\end{itemize}

Where  $C=C(p_1,p_2,\Omega,C_{log},N)$ and $\gamma$ and $\beta$
are given in Definitions \ref{p*} and \ref{p_*}, respectively.
\end{corol}
\begin{proof}
It follows by Lemma \ref{globalth}, taking $\displaystyle
v={u}|u|_{{W}^{1,p(\cdot)}(\mathcal{T}_h)}^{-1}$.
\end{proof}

\begin{remark}\label{rem25} Under the same hypothesis of the last corollary,
    if  $1\le q\le p^*$ in ${\O}$,  we have that,
    for all $u\in  S^k(\th),$
$$
    \|u-Q_h(u)\|_{\lq} \le C h^{N(1-\gamma)}
|u|_{{W}^{1,p(\cdot)}(\mathcal{T}_h)} \quad \mbox{ and } \quad
\|\nabla Q_h(u)\|_{\lp}\le C|u|_{{W}^{1,p(\cdot)}(\mathcal{T}_h)},
$$
where  $C=C(p_1,p_2,\Omega,C_{log},N).$
\end{remark}

\section{The lifting operator}\label{lifop}

We begin this section by defining, as in \cite{BO} (see also
\cite{ABCM}), the lifting operator, i.e.
\begin{defi}\label{lifting}
Let $l\ge0$ and
 $R_h\colon\wpt\to S^l(\th)^N$ defined as,
 $$\int_{\O} \langle R_h(u), \phi\rangle\, dx =
-\int_{\Gamma_{int}} \langle \[u\], \{\phi\} \rangle \, dS
 \quad \forall \phi\in S^l(\th)^N.$$
\end{defi}

This operator appears in the first term of the discretized
functional $I_h$. As we can see from the definition, this operator
represents the contribution of the jumps to the distributional
gradient. This is the reason  why it is crucial to add  this term
in order to have the consistency of the method.

We point out that this lifting operator was first used in
\cite{BR1997} in order to describe the contributions of  the jumps
across the interelements of the computed solution on the
(computed) gradient of the solution in a mixed formulation of
Navier-Stokes equations. It was also used in \cite{BMMPR2000}
where a solid mathematical background for the method introduced in
\cite{BR1997} was proposed.

\medskip

Now, we give a bound of the $\lp$-norm of $R_h(u)$ in terms of the
jumps of $u$ in $\gi$.

 When $p$ is constant the proof follows from an inf-sup condition. Since in our case, we are dealing with the Luxemburg norm,
  we can't prove the boundedness directly from the definition.
  We can prove this inf-sup condition, but we can not use it to prove the result.
 Instead, we find a local characterization of $R_h$ in order to prove a local bound and then, we  prove the global bound.

We give first the local estimate.

\begin{lema}\label{infcontloc} There exists a constant $C_1$ such that, for any $\kappa\in\mathcal{T}_h,$ we have
$$
\|R_h(u)\|_{L^{p(\cdot)}(\kappa)}\leq
C\|{\bf{h}}^{-1/p'(x)}\[u\]\|_{L^{p(\cdot)}(\kappa\cap\Gamma_{int})}\quad
\forall u\in\wpt\quad \forall h\in (0,1].
$$
\end{lema}
\begin{proof}
We proceed in two steps.

{\it Step 1}. We first want to prove that,

\begin{equation}\label{cotaruk}
|R_h(u)|\leq \frac{C}{h_{\kappa}^N} \sum_{e\subset\kappa} \int_e
|\[u\]|\, dS\quad \forall \kappa\in\th\end{equation} where $C$ is
independent of $\kappa$ and $h$.

We begin by observing that, by Hypothesis \ref{omega}, there
exists $m=m(k,N)\in\mathbb{N}$ such that for each $\kappa\in\th$,
$$R_h(u)|_{\kappa}\circ F_{\kappa}=\sum_{i=1}^m a_i \varphi_i(x),$$
where $\{\varphi_i\}$ is the standard nodal base of
$\left(P^l\right)^N$ in the reference element
$\hat{\kappa}:=F^{-1}_{\kappa}(\kappa)$.

Using the definition of $R_h$ we have that for each $1\leq j \leq
m$,
$$
\int_{\O} R_h(u) \varphi_j\circ F^{-1}_{\kappa}(x)\,
dx=\sum_{i=1}^m a_i \int_{\kappa} \varphi_i\circ
F^{-1}_{\kappa}(x)\varphi_j\circ F^{-1}_{\kappa}(x)\, dx
=-\sum_{e\subset \kappa} \int_e \[u\]\{\varphi_j\circ
F^{-1}_{\kappa}(x)\}\, dS.
$$
On the other hand, if we change variables and we use Hypothesis
\ref{mesh} and the fact that
 $|\varphi_i(x)|\leq 1$, we get
$$
\int_{\kappa}\varphi_i\circ F^{-1}_{\kappa}(x)\varphi_j\circ
F^{-1}_{\kappa}(x)\, dx=h_{\kappa}^N\int_{\hat{\kappa}}
\varphi_i(x)\varphi_j(x) \frac{|\det(DF_{\kappa})|}
{{h_\kappa}^N}\, dx=h_{\kappa}^N d_{ij}
$$
with $d_{ij} \sim 1.$

Therefore,
$$
R_h(u)|_{\kappa}\circ F_{\kappa}=\frac{1}{h_\kappa^N}\sum_{i=1}^m
(D^{-1}b)_i\varphi_i(x)\, dx,
$$
where $D=(d_{ij})$ and $\di b_j=-\sum_{e\subset \kappa} \int_e
\[u\] \{\varphi_j\circ F^{-1}_{\kappa}(x)\}\, dS$.

Thus, using that $|\varphi_i(x)|\leq 1$, we arrive at
\eqref{cotaruk}.

\medskip

{\it Step 2}. Now, we show that there exists a constant $C_1$ such
that, for any $\kappa\in\mathcal{T}_h,$ we have
$$
\|R_h(u)\|_{L^{p(\cdot)}(\kappa)}\leq
C\|{\bf{h}}^{-1/p'(x)}\[u\]\|_{L^{p(\cdot)}(\kappa\cap\Gamma_{int})}\quad
\forall u\in\wpt\quad \forall h\in (0,1].
$$

 By inequality \eqref{holderene}, we have
 $$
 \int_e |\[u\]|\, ds\leq C h_{e}^{N(1-\frac{1}{p_-})}
 \|\[u\]h_e^{-\frac{1}{p'(x)}}\|_{L^{p(\cdot)}(e)}.
 $$
 Thus, by Hypothesis \ref{mesh} and \eqref{cotaruk}, we have that
 $$
 |R_h(u)|\leq  \frac{C}{h_{\kappa}^{{N}/{p_-}}} \sum_{e\subset\kappa}
 \|\[u\]h_e^{-\frac{1}{p'(x)}}\|_{L^{p(\cdot)}(e)}.
 $$
Now, take $T=\sum_{e\in \kappa}
\|\[u\]h_e^{-\frac{1}{p'(x)}}\|_{L^{p(\cdot)}(e)}.$ Then,
$$
\int_{\kappa} \Big|\frac{R_h(u)}{T}\Big|^{p(x)}\, dx \leq C
\int_{\kappa} {h_{\kappa}^{{-Np(x)}/{p_-}}}\, dx\leq C
h_{\kappa}^{N(1-p_+/p_-)}\leq C
$$
where in the last inequality we are using Proposition \ref{papa}.

The result follows now by Remark \ref{cardinal}.
\end{proof}

\begin{lema}\label{lifcont}
Let $p:\O\to\colon[1,\infty)$ be log-H\"{o}lder continuous in
$\O$. Then, there exists a constant $C$ such that,
$$\|R_h(u)\|_{\lp}\leq C\|{\bf{h}}^{-1/p'(x)}\[u\]\|_{\lpi}\quad \forall u\in\wpt\quad \forall h\in (0,1].$$
\end{lema}

\begin{proof}
First, if we assume that $\|{\bf{h}}^{-1/p'(x)}\[u\]\|_{\lpi}\le
1$, we can prove using Lemma \ref{infcontloc} and proceeding as in
Lemma \ref{globalth} that,
$$\int_{\O} |R_h(u)|^{p(x)}\, dx\le C. $$

Finally, taking  $v=u
\left(\|{\bf{h}}^{-1/p'(x)}\[u\]\|_{\lpi}\right)^{-1}$, we obtain
the desired result.
\end{proof}

\section{Convergence of the method}\label{cm2}

In this section we first prove the  broken Poincar\'e Sobolev
inequality which is crucial to get  compactness. We also prove the
coercivity of the functional and we finally arrive at the proof of
Theorem \ref{conv}.

\begin{teo}\label{brok-poncare}
    Let $p:\O\to[1,+\infty)$ be $\log-$H\"{o}lder continuous in
    $\overline{\O}.$
    There exists a constant $C$ such that,
    $$
    \|u-(u)_{\O}\|_{\lpe}\leq C |u|_{\wpt}\quad\forall u\in S^k(\th)
    \quad\forall h\in(0,1],$$
    where $(u)_{\O}=\frac1{|\O|}\int_{\Omega} u \, dx.$
    In particular,
    $$
    \|u\|_{\lpe}\leq C \Big(\|u\|_{L^1(\O)}+|u|_{\wpt}\Big)\quad
    \forall u\in S^k(\th) \quad\forall h\in(0,1].$$
\end{teo}

\begin{proof}
We begin by observing that
$$
\|u-(u)_{\O}\|_{\lpe}\le\|u-Q_h(u)\|_{\lpe}+\|Q_h(u)-(Q_h(u))_{\O}\|_{\lpe}+
C\|Q_h(u)-u\|_{L^1(\O)}.
$$
Then, using the Remark \ref{rem25} and Theorem \ref{poinc2}, we
have
    $$
    \|u-(u)_{\O}\|_{\lpe}\leq C |u|_{\wpt}\quad\forall u\in S^k(\th)
    \quad\forall h\in(0,1].
    $$
    The proof is complete.
\end{proof}

\begin{teo}\label{cm}
For each $h\in(0,1],$ let  $u_h\in \wpt$. If there exists a
constant $C$ independent of $h$ such that ${I}_h(u_h)\le C$ for
all $h\in (0,1]$, then
$$
\sup_{h\in(0,1]}\left(\|u_h\|_{L^1(\Omega)}+|u_h|_{\wpt}\right)<\infty.
$$
Moreover,
$$\sup_{h\in(0,1]}\int_{\partial\O}|u_h-u_D|^{p(x)}{\bf{h}}^{1-p(x)}\,dS<\infty.$$

\end{teo}
\begin{proof}
Since $I_h(u_h)\le C$ then,
$\|{\bf{h}}^{-1/p'(x)}\[u_h\]\|_{L^{p(\cdot)}(\Gamma_{int})}\le
C$. And, by Lemma \ref{lifcont} and Proposition \ref{equi}, we
have
$$\int_{\O}|R_h(u_h)|^{p(x)}\, dx \le C.$$
Using the third inequality in Proposition \ref{desip} and the
above inequality, we get
\begin{align*}\int_{\O}|R_h(u_h)+\nabla u_h|^{p(x)}\, dx
\ge 2^{1-p_2} \int_{\O}|\nabla u_h|^{p(x)}\, dx-C.\end{align*}
Therefore,
$$I_h(u_h)+C\ge 2^{1-p_2} \int_{\O}|\nabla u_h|^{p(x)}\, dx
+\int_{\Gamma_D}|u_h-u_D|^{p(x)}{\bf{h}}^{1-p(x)}\,dS
+\int_{\Gamma_{int}}|\[u_h\]|^{p(x)}{\bf{h}}^{1-p(x)}\, dS.$$
Thus, as ${I}_h(u_h)\le C,$ we obtain that $\di|u_h|_{\wpt}$ and
$\di\int_{\partial\O}|u_h-u_D|^{p(x)}{\bf{h}}^{1-p(x)}\,dS$ are
uniformly bounded.

Finally, by Friedrichs inequality for BV, Lemma \ref{cotaD},
H\"{o}lder inequality, Proposition \ref{equi} and the fact that
${\bf h}\leq 1$ we have,
\begin{align*}
\|u_h\|_{L^1(\O)}\leq & C
\left(|u_h|_{\wpt}+\int_{\Gamma_D}|u_h|\, dS\right)\\ \leq
 & C \left(|u_h|_{\wpt}+\int_{\Gamma_D}|u_D|\, dS+\|(u_h-u_D) {\bf{h}}^{-1/p'(x)}\|_{\lpd}\|{\bf{h}}^{1/p'(x)}\|_{L^{p'(\cdot)}(\Gamma_D)}\right)
 \\ \leq
 & C \left(|u_h|_{\wpt}+\int_{\Gamma_D}|u_D|\, dS+\|(u_h-u_D) {\bf{h}}^{-1/p'(x)}\|_{\lpd}\right).
 \end{align*}

This completes the proof.
\end{proof}



\begin{lema}\label{lqr}
 Let  $p,$ $s$ and $t$ be functions satisfying (H1), (H2) and (H3) respectively. Let $u_h\in
 S^k(\th)$ be under the
conditions of Theorem \ref{cm}. Then, there exist a sequence
$h_j\to0$ and a function $u\in\wp$ such  that
\begin{align}
    \label{cbv}u_{h_j}&\stackrel{*}{\rightharpoonup} u\quad\ \ \mbox{weakly* in } BV(\O)\\
    \label{crh}\nabla u_{h_j}+R_h(u_{h_j})&\rightharpoonup \nabla u \quad\mbox{weakly in } \lp,\\
    \label{lq2}u_{h_j}&\rightarrow u\quad\ \ \mbox{strongly in } L^{s(\cdot)}(\O),\\
    \label{lqb} u_{h_j}&\rightarrow   u\quad\ \ \mbox{strongly in } L^{t(\cdot)}(\partial\O).
\end{align}
\end{lema}
\begin{proof}
    We first observe that, by Theorem \ref{cm}, we have that
     $$
     \sup_{h\in(0,1]}(\|u_h\|_{L^1(\Omega)}+|u_h|_{\wpt})<\infty \quad \forall h\in(0,1].
     $$
    Then, the proofs of \eqref{cbv} and \eqref{crh} follows by applying Theorem 5.2 in \cite{BO} with value $p_1$, and  using that $R_h(u_h)$ is
    bounded in $\lp,$ see Lemma \ref{lifcont}.

    We now prove \eqref{lq2}. By \eqref{cbv} and the compactness of the embedding $BV(\O)\subset L^1(\O),$ there
exists a subsequence of $u_{h_j},$ still denoted by $u_{h_j},$
such that
$$u_{h_j}\to u \mbox{ in }L^1(\O).$$
Since $\|u_{h_j}\|_{L^1}+|u_{h_j}|_{\wpt}$ is bounded, by Theorem
\ref{brok-poncare}, $\|u_{h_j}\|_{\lpe}$ is bounded, and by
Theorem \ref{embed}, $u\in\wp\subset \lpe$. Therefore, using
Theorem \ref{interp}, we obtain that
\begin{equation}
    u_{h_j}\to u \quad \mbox{ in } L^{s(\cdot)}(\O),
    \label{auxx0}
\end{equation}
for all $s$ satisfying (H2).

Finally, we prove \eqref{lqb}. We begin by observing that, by
Corollary \ref{globalthfin},
\begin{equation}
    \|u_h-Q_h(u_h)\|_{\lp}\to 0 \mbox{ as } h\to0
    \label{auxx1}
\end{equation}
and $\{Q_h( u_h)\}$ is bounded in $\wp.$ Then, there exists
$v\in\wp$ and subsequence $\{Q_{h_j}( u_{h_j})\}$ such that
$Q_{h_j}(u_{h_j}) \rightharpoonup v  $ weakly in $\wp.$ Therefore,
by \eqref{auxx0} and \eqref{auxx1}, $v=u.$

Using Theorem \ref{traza},
\begin{equation}
Q_{h_j}(u_{h_j})\to u \mbox{ strongly in }
L^{t(\cdot)}(\partial\O).
    \label{auxx2}
\end{equation}

Now, taking $\bar{t}:\O\to [1,\infty)$  log-H\"{o}lder with $t\leq \bar{t}<p_*$ and by
Corollary \ref{globalthfin},  we get
$Q_{h_j}(u_{h_j})-u_{h_j}\to 0$ strongly in
$L^{t(\cdot)}(\partial\O).$  Therefore, $u_{h_j}\to u$ in
$L^{t(\cdot)}(\partial\O)$.
\end{proof}

\medskip

Before  proving the convergence of the minimizers, we need an
auxiliary lemma. It is in  this step  where we need more
regularity of the boundary data.

\begin{lema}\label{auxiliar}
Let $h\in (0,1]$, and $p:\Omega\to(1,\infty)$  satisfying (H1).
Assume that $u_D\in W^{2,p_2}(\O)$ and let $v\in
W^{2,p_2}(\O)\cap\A$ then, there exists $v_h\in U^1(\th)$, such
that
$$
\di \|v_h-v\|_{\wp}\to 0 \quad \mbox{ as } h\to 0,
$$
and $$I_h(v_h)\to I(v)\quad \mbox{ as } h\to 0.$$
\end{lema}
\begin{proof}
Since $p$ is log-H\"{o}lder, we have that
$C^{\infty}(\bar{\Omega})$ are dense in $\wp$. Then the first part
follows  by standard approximation theory, see in Theorem 3.1.5
\cite{Ci}.

Moreover, $v_h$ satisfies
\begin{equation}\label{interpol}
\|v-v_h\|_{L^{p_2}(\partial\kappa)}\leq C
\|v-v_h\|_{W^{1,p_2}(\kappa)}\leq C h_{\kappa} \|D^2
v\|_{L^{p_2}(\kappa)}\end{equation} for each $\kappa\in\th$.
Using Remark \ref{comp} and  summing over all $e\in \partial\O,$
we have
\begin{equation}\label{desip2}\int_{\partial\O} |v-v_h|^{p_2} {\bf{h}}^{1-p_2}\, ds \leq Ch \|D^2 v\|^{p_2}_{L^{p_2}(\O)}.\end{equation}
In addition, by H\"{o}lder inequality and since $h\leq 1$, we have
$$\int_{\partial\O}|v-v_h|^{p(x)} {\bf{h}}^{1-p(x)}\, ds \leq C\| |v-v_h|^{p(\cdot)} {\bf{h}}^{(1-p_2)p(\cdot)/p_2}\|_{L^{p_2/p(\cdot)}(\partial\O)}.$$
Since
$$ \int_{\partial\O}(|v-v_h|^{p(x)} {\bf{h}}^{(1-p_2)p(x)/p_2})^{p_2/p(x)}\, ds=
\int_{\partial\O}|v-v_h|^{p_2} {\bf{h}}^{(1-p_2)}\ ds,$$ then, by
\eqref{desip2}
$$
\int_{\partial\O}|v-v_h|^{p(x)} {\bf{h}}^{1-p(x)}\, ds\to 0 \quad
\mbox{ as } h\to 0.$$ Since $v_h\in \wp$ then $\[v_h\]=0$ and
$R_h(v_h)=0$. Finally, using \eqref{interpol} and Theorem
\ref{debil}, we obtain the desired result.
\end{proof}

Now, we are in a condition to prove Theorem \ref{conv}.

\begin{proof}[Proof of Theorem \ref{conv}]
By Lemma \ref{auxiliar} there exist $w_h\in  U^1(\th)$ such that
 $w_h\to u_D$ strongly in $\wp$ and
 $I_h(w_h)\to I(u_D)$. Therefore, since $I_h(u_h)\leq I_h(w_h)$, we have that  $I_h(u_h)$ is bounded.

Then, by Theorem \ref{cm} and Lemma \ref{lqr} there exist a
subsequence $u_{h_j}$ and $u\in \wp$ such that
\begin{equation}\label{limites}\begin{aligned}
u_{h_j}&\stackrel{*}{\rightharpoonup} u\quad\ \ \mbox{weakly* in } BV(\O),\\
\nabla u_{h_j}+R_h(u_{h_j})&\rightharpoonup \nabla u \quad\mbox{weakly in } \lp,\\
 u_{h_j}&\to u  \quad\mbox{strongly in } L^{s(\cdot)}(\O),\quad \forall s\ \mbox{satisfying } (H2)\\
  u_{h_j}&\to u  \quad\mbox{strongly in } L^{t(\cdot)}(\partial\O), \quad \forall t\ \mbox{satisfying } (H3).
\end{aligned}\end{equation}

On the other hand, since the penalty term,
$$\int_{\Gamma_D} {\bf{h}}^{1-p} |u_h-u_D|^p\, dS$$
is bounded, we have that
$$\|u-u_D\|_{\lpbg}\leq \|u-u_{h_j}\|_{\lpbg}+\|u_{h_j}-u_D\|_{\lpbg}\to 0. $$
 Then $u\in\mathcal{A}$.

Taking $s=q$ and $t=r$ in \eqref{limites}, by Proposition \ref{debil} we have
 \begin{equation}\label{liminf2}
\begin{array}{lll}
    I(u)&\leq &\displaystyle
\liminf_{j\to\infty}\left[\int_{\O}
     \left(|\nabla u_{h_j}+R_h(u_{h_j})|^{p(x)}+|u_{h_j}-\xi|^{q(x)}\right)\, dx
+\int_{\Gamma_N} |u_{h_j}|^{r(x)}\, dS\right]\\[8pt]
&\le&\displaystyle
 \liminf_{j\to\infty}I_{h_j}(u_{h_j})\le \limsup_{j\to\infty}I_{h_j}(u_{h_j}).
 \end{array}
\end{equation}

Now, we want to prove that $u$ is the minimizer of $I$. Let
$v\in\A\cap W^{2,p_2}(\O)$, and let $v_h\in U^1(\th)$ as in Lemma
\ref{auxiliar}. Then  $I_h(v_h)\to I(v)$. Therefore, by
\eqref{liminf2}
\begin{equation}\label{loquesale}I(u)\leq \liminf_{j\to\infty}I_{h_j}(u_{h_j})\le \limsup_{j\to\infty}I_{h_j}(u_{h_j})\le \lim_{j\to\infty}I_{h_j}(v_{h_j})=I(v).\end{equation}

Now, let $w\in  \A$, then for any $\ep>0$ there exists $v\in\A\cap
W^{2,p_2}(\O)$ such that $\|v-w\|_{\wp}<\ep$. By Theorem
\ref{debil} we have that $I(v)<I(w)+\ep$, therefore by
\eqref{loquesale}
$$I(u)\leq \liminf_{j\to\infty}I_{h_j}(u_{h_j})\le \limsup_{j\to\infty}I_{h_j}(u_{h_j})\le I(w)+\ep.$$
Taking $\ep\to 0$, we get
$$I(u)\leq \liminf_{j\to\infty}I_{h_j}(u_{h_j})\le \limsup_{j\to\infty}I_{h_j}(u_{h_j})\le I(w)\quad \forall\, w \in \A.$$
 Therefore $I(u)\leq I(w)$.

 Moreover, taking $w=u,$ we have that all
 the inequalities in \eqref{liminf2} are equalities, therefore
we have that $I_{h_j}(u_{h_j})\to I(u)$. Then
$$\int_{\Gamma_{int}}|\[u_{h_j}\]|^{p(x)}{\bf{h_j}}^{1-p(x)}\, dS \rightarrow  0$$ and using Lemma \ref{lifcont}
we  have that $R_h(u_{h_j})\to 0$. This fact and \eqref{limites}
imply that $\nabla u_{h_j} \rightharpoonup \nabla u$ weakly in
$\lp$.

Since $u$ is the unique minimizer of $I,$  the whole sequence
$u_h$ converges to $u.$

Finally, since
$$\nabla u_h+R_h(u_h) \rightharpoonup\nabla u \mbox{ weakly in } \lp \mbox{ and }\int_{\O}(|\nabla u_h+R_h(u_h)|^{p(x)}\, dx \to \int_{\O}|\nabla u|^{p(x)}\, dx,$$
by Proposition  \ref{debil}, $\nabla u_h+R_h(u_h)\to \nabla u$
strongly in $\lp$. Therefore, since $R_h(u_h)\to 0$ strongly in
$\lp$, we get that $\nabla u_h\to \nabla u$ strongly in $\lp$.
\end{proof}

\section{The Continuous Galerkin Method}\label{confor}
In order to make a complete study of this problem, we prove the
convergence of the Continuous Galerkin finite element method for
our problem. In the next section, we make a comparison of the two
methods in an example.

For simplicity, we take the following functional:
$$
I(u)=\int_{\O}\left(\frac{|\nabla u|^{p(x)}}{p(x)}
+\frac{|u-\xi|^{q(x)}}{q(x)}\right)\,dx$$ with $q(x)$ satisfying
(H2). Then, since the functional $I$ is strictly convex and
coercive in $\A$ there exists a unique minimizer of the problem.

We take now a partition of $\O$ as in Hypothesis  \ref{omega} and
the usual conforming subspace  $U^k(\th)$  of $\wp$. This subspace
consists of all continuous functions such that they are
polynomials of degree at most $k$ in each $\kappa\in\th$ . We
assume that for some $h'$, $u_D\in U^k(\thp)$ (this assumption
replaces the one on Lemma \ref{auxiliar}).

Let now $h\leq h'$ and
$$
V^k_h=\{v_h\in U^k(\th) \colon v_h=u_D\mbox{ on } \partial\O\}.
$$
For simplicity, we  may assume that $h'=1$.

\begin{remark}\label{remsu}
Let  $\Pi_h:C_0^{\infty}(\O)\to U_h^k$ be the interpolant mapping
defined in Theorem 3.1.5 in \cite{Ci}. Then,we have that
$$
\|\Pi_h\phi-\phi\|_{W^{1,p(\cdot)}(\O)}\leq C
\|\Pi_h\phi-\phi\|_{W^{1,p_2}(\O)}\to 0,
$$
for any $\phi\in C_0^{\infty}({\O}).$ We also have, by the
continuity of $I,$ that  $I(\Pi_h\phi+u_D)\to I(\phi+u_D)$ as
$h\to 0.$
\end{remark}

By the strict convexity of $I$, for each $h\in(0,1]$ there exists
a function $u_h\in V_h^k$ such that $u_h$ is a minimizer in
$V^k_h$ of $I$.

\medskip

Now we prove the main result of this section.

\begin{teo}
    The sequence $\{u_h\}$ converges to $u$ strongly in $\wp$, where $u$ is the unique minimizer of $I$.
\end{teo}

\begin{proof}
    Since $\{I(u_h)\}$ is uniformly bounded, there exists a subsequence of $\{u_h\}$ (still denoted by $\{u_h\}$) such that
    \begin{align}
        \label{weakconcon} u_{h}\rightharpoonup u &\mbox{ weakly in } \wp,\\
        \label{strongconcon} u_{h}\to u &\mbox{ strongly in } \lp.
    \end{align}

   As in the proof of Theorem \ref{conv}, using Remark \ref{remsu} instead of Lemma \ref{auxiliar}, we can prove that $u$ is  the minimizer of $I$ and that
   $I(u_h)\to I(u)$ as $h\to0.$ By the convexity of $I$ and \eqref{weakconcon}, we have that
   \begin{equation}
       \int_{\O}|\nabla u_h|^{p(x)}\ dx\to \int_{\O}|\nabla u|^{p(x)} \ dx \quad \mbox{as } h\to 0.
       \label{totalaux}
   \end{equation}
   Then, by \eqref{strongconcon} and \eqref{totalaux}, using Proposition \ref{debil} (3), we have that $\{u_h\}$ converges to $u$ strongly in $\wp$.

   Observe that, as in Theorem \ref{conv}, we can conclude that the whole sequence $\{u_h\}$ converges to $u$ strongly in $\wp.$
 \end{proof}

\section{One dimensional example}\label{examples}

In this section, we give an example in one dimension. Our idea is
to compare the Continuous Galerkin Finite Element Method (CGFEM)
versus the Discontinuous Galerkin Finite Element Method (DGFEM).
We will see in the following example that, if the function $p$
attains values close to  one,  our method converges faster to the
solution.

\medskip

Let $\O=(-1,1)$, $0<\ep,a<1$ and $p:[-1,1]\to [1,2]$ given by
 $$p(x)= \begin{cases}\frac{1-\ep}{a}|x|+1+\ep\quad& \mbox{ if
     } \quad \quad \quad |x|\leq a,\\[5pt]
 \quad \quad2 \quad& \mbox{ if }\quad  a\leq |x|\leq 1.
\end{cases}$$
For this function $p(x)$ and for a given $B>0,$ we study the
following problem,
\begin{equation}\label{dim1}\begin{cases} (|u'(x)|^{p(x)-1}u'(x))'=0
    \quad&\mbox{ in } (-1,1),\\
u(1)=-u(-1)=B.&\\ \end{cases}\end{equation}

We begin by observing that, since the operator is strictly
monotone, we have a unique solution of \eqref{dim1}. Moreover, the
solution satisfies  $|u'(x)|=C^{\frac{1}{(p(x)-1)}}$ for some
constant $C>0$. Therefore,  $|u'(x)|>\max\{C,C^{1/\ep}\}$  and,
using that $u\in C^{1,\alpha}([-1,1]),$ we have that $u'$ does not
change sign. Then, since $u(1)>u(-1)$ we obtain that $u'(x)>0$.

Thus,
\begin{align}\label{ulineal}
    u(x)=C(x+1)-B \quad&\mbox{ if } -1\leq x \leq
    -a,\\ \label{ulineal2}
    u(x)=C(x-1)+B \quad &\mbox{ if }\quad a\leq x \leq 1.
\end{align}
Since $p$ is even, we have that $u$ is odd, so $u(0)=0$.
Therefore, $u(x)=\int_0^x C^{\frac{1}{p(s)-1}}\, ds$ for all $x\in
[-a,a]$.

On the other hand, since the derivative of $u$ at zero has modulus
$C^{1/\ep}$, if $C>1$ we have
$$
\lim_{\ep\to0}|u'(0)|=+\infty.
$$
This is reasonable since we expect to have  big derivative when
$p$ approaches the value one.

\medskip

From now on, we take  $\ep=a=.01$ and since it is easier to get
$B$ from $C$,
we impose $C=1.3$. 
Then $B=\int_0^1 1.3^{\frac{100}{1+999s}}\, ds \simeq 1.03\ 10^6.$
Observe that in this case,
 $|u'(0)|=1.3^{100}\simeq 2.4 \times 10^{11}.$

\medskip

Now, we find the corresponding solution for the CGFEM and the
DGFEM. In both cases we take a uniform partition of $[-1,1]$ in
$n$ subintervals with size $2/n$ and $k=1$. We use the trapezoidal
numerical quadrature to compute the integrals appearing in the
discrete functionals. The analysis of these integration errors
falls beyond the scope of this paper.

Observe that for the continuous method, we impose the boundary
conditions and then, the space where we find minimizers has
dimension  $n-2$. For the Discontinuous method, since we do not
impose conditions on the boundary, and the number of nodal basis
are $2n-2$, we are minimizing in a space of this dimension.
Therefore, to make a comparison between both methods, we  compare
the discrete problem for  the DGFEM in $n-$intervals with the
CGFEM in $2n-$intervals.

We want to mention that, in order to find minimizers of both
discrete problems,  we use a BFGS Quasi-Newton method (see
\cite{Go, Sh}).

In the next two figures, we first plot the solution versus the
approximation using  the DGFEM and CGFEM for the case $n=41.$ The
second figure is the graphic of the function $p(x)$.

\centerline{\includegraphics[width=15cm]{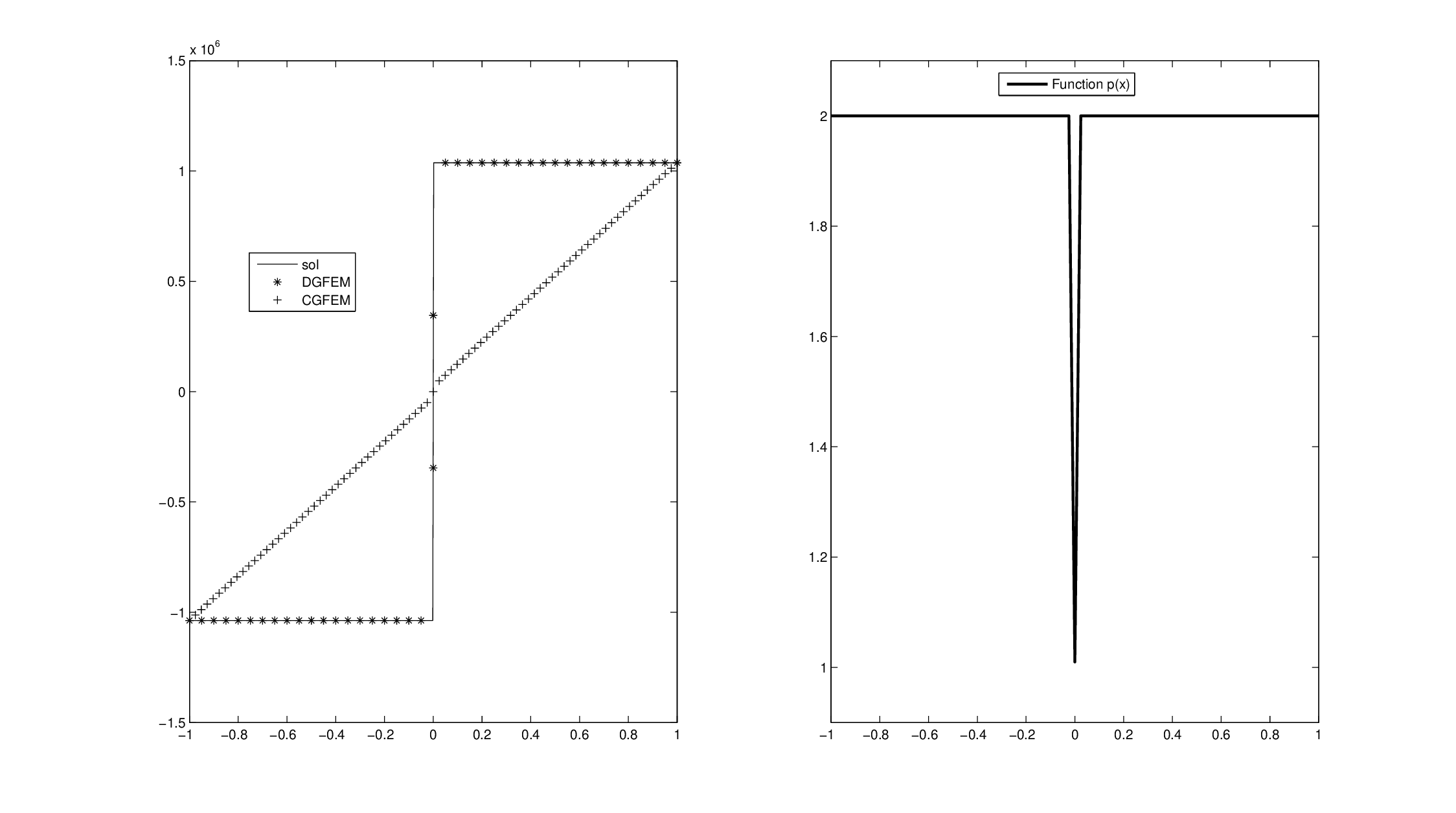}}

Note that, when we use the CGFEM the discrete solution is close to
the function $y=x$ which is a solution of \eqref{dim1} with
$p\equiv 2$, that means that this method needs a smaller step in
order to see the points where $p$ is close to one.

In the following figure we can see that, the minimizers of the
continuous methods are far from the solution even for $n=150$
(300-intervals). We need $n=200$ (400-intervals) to arrive at a
good approximation of $u$.

\centerline{\includegraphics[width=15cm]{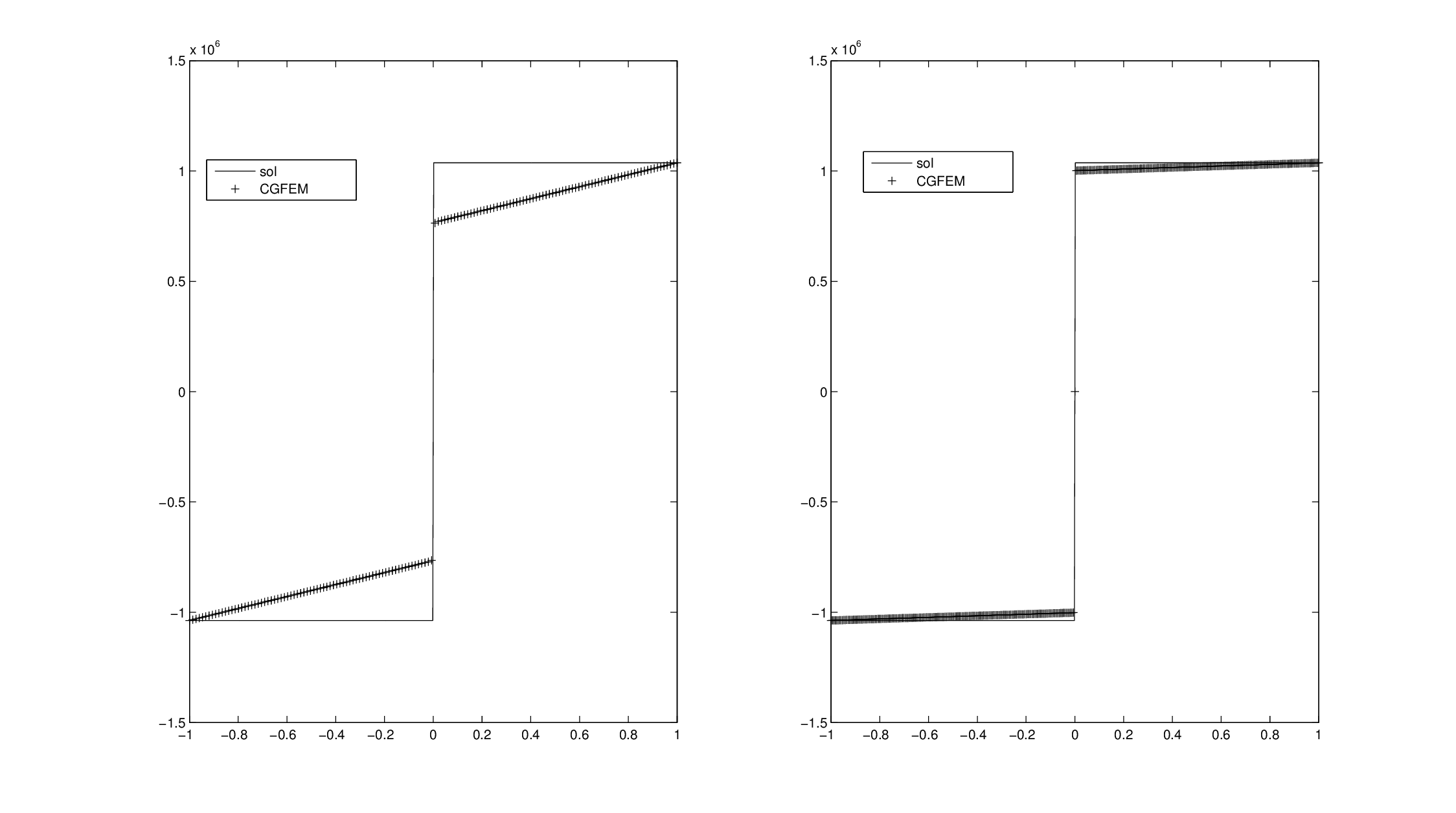}}

\begin{ack}
This work had been improve thanks to the suggestions and comments
of one of the referees. The authors want to thank  him strongly.
We also like to thank N. Wolanski  for her throughout reading of
the manuscript that help us to improve the presentation of the
paper.
\end{ack}

\def\cprime{$'$} \def\ocirc#1{\ifmmode\setbox0=\hbox{$#1$}\dimen0=\ht0
  \advance\dimen0 by1pt\rlap{\hbox to\wd0{\hss\raise\dimen0
  \hbox{\hskip.2em$\scriptscriptstyle\circ$}\hss}}#1\else {\accent"17 #1}\fi}
\providecommand{\bysame}{\leavevmode\hbox
to3em{\hrulefill}\thinspace}
\providecommand{\MR}{\relax\ifhmode\unskip\space\fi MR }
\providecommand{\MRhref}[2]{%
  \href{http://www.ams.org/mathscinet-getitem?mr=#1}{#2}
} \providecommand{\href}[2]{#2}

\bibliographystyle{amsplain}

\begin{thebibliography}{10}

\bibitem{ABCM}
D.N Arnold, F.~Brezzi, B.~Cockburn and Marini, \emph{Unified
analysis of
  discontinuous galerkin methods for elliptic problems}, SIAM J. Num. Anal, \textbf{39}
  (2002), 1749--1779.

\bibitem{BR1997} F. Bassi and S. Rebay, \emph{A High-Order Accurate Discontinuous Finite Element Method
for the Numerical Solution of the Compressible Navier–Stokes
Equations}, J. Comput. Phys. \textbf{131} (1997), 267--279.

\bibitem{BCE}
E. M. Bollt, R. Chartrand, S. Esedo{\=g}lu, P. Schultz and K. R.
  Vixie, \emph{Graduated adaptive image denoising: local compromise between
  total variation and isotropic diffusion}, Adv. Comput. Math. \textbf{31}
  (2009), no.~1-3, 61--85.

\bibitem{BR}
S. C. Brenner, \emph{Poincar\'e-{F}riedrichs inequalities for
piecewise
  {$H^1$} functions}, SIAM J. Numer. Anal. \textbf{41} (2003), no.~1, 306--324
  (electronic).

\bibitem{BMMPR2000} F. Brezzi, G. Manzini, L. D. Marini, P. Pietra and A. Russo,
\emph{Discontinuous Galerkin approximations for elliptic
problems}, Numer. Meth. Partial Diff. Eq., \textbf{16} (2000), pp.
365--378.

\bibitem{BO}
A. Buffa and Ch. Ortner, \emph{Compact embeddings of broken
  {S}obolev spaces and applications}, IMA J. Numer. Anal. \textbf{29} (2009),
  no.~4, 827--855.

\bibitem{CL}
A. Chambolle and P. L. Lions, \emph{Image recovery via total
  variation minimization and related problems}, Numer. Math. \textbf{76}
  (1997), no.~2, 167--188.

\bibitem{CLR}
Y. Chen, S. Levine and M. Rao, \emph{Variable exponent, linear
  growth functionals in image restoration}, SIAM J. Appl. Math. \textbf{66}
  (2006), no.~4, 1383--1406 (electronic).

\bibitem{Ci}
Ph. Ciarlet, \emph{The finite element method for elliptic
problems}, vol.~68,
  North-Holland, Amsterdam, 1978.

\bibitem{Di}
L. Diening, \emph{Theoretical and numerical results for
electrorheological
  fluids,}, Ph.D. thesis, University of Freiburg, Germany (2002).

\bibitem{D}
\bysame, \emph{Maximal function on generalized {L}ebesgue spaces
{${L}\sp
  {p(\cdot)}$}}, Math. Inequal. Appl. \textbf{7} (2004), no.~2, 245--253.

\bibitem{D3}
\bysame, \emph{Riesz potential and {S}obolev embeddings on
generalized
  {L}ebesgue and {S}obolev spaces {$L^{p(\cdot)}$} and {$W^{k,p(\cdot)}$}},
  Math. Nachr. \textbf{268} (2004), 31--43.

\bibitem{DHHR}
L. Diening, P. Harjulehto, P. H\"{a}st\"{o} and M. Ruzicka,
\emph{Lebesgue and
  sobolev spaces with variable exponents}, Lecture Notes in Mathematics, vol.
  2017, Springer-Verlag, New York, 2011.

\bibitem{DHN}
L. Diening, P. H{\"a}st{\"o} and A. Nekvinda, \emph{Open problems
in variable
  exponent {L}ebesgue and {S}obolev spaces}, Function Spaces, Differential
  Operators and Nonlinear Analysis, Milovy, Math. Inst. Acad. Sci. Czech
  Republic, Praha, 2005.

\bibitem{FZ}
X. Fan and Q. H. Zhang, \emph{Existence of solutions for
  {$p(x)$}-{L}aplacian {D}irichlet problem}, Nonlinear Anal. \textbf{52}
  (2003), no.~8, 1843--1852.

\bibitem{Fanx2}
X. Fan, \emph{Regularity of nonstandard {L}agrangians
{$f(x,\xi)$}},
  Nonlinear Anal. \textbf{27} (1996), no.~6, 669--678.

\bibitem{Fanx}
\bysame, \emph{Boundary trace embedding theorems for variable
exponent
  {S}obolev spaces}, J. Math. Anal. Appl. \textbf{339} (2008), no.~2,
  1395--1412.

\bibitem{FS}
X. Fan and D. Zhao, \emph{On the spaces {$L\sp {p(x)}(\Omega)$}
and
  {$W\sp {m,p(x)}(\Omega)$}}, J. Math. Anal. Appl. \textbf{263} (2001), no.~2,
  424--446.

\bibitem{Go}
D. Goldfarb, \emph{A family of variable metric updates derived by
variational
  means},  \textbf{24} (1970), 23--26.

\bibitem{HH}
P. Harjulehto and P. H{\"a}st{\"o}, \emph{A capacity approach to
the
  {P}oincar\'e inequality and {S}obolev imbeddings in variable exponent
  {S}obolev spaces}, Rev. Mat. Complut. \textbf{17} (2004), no.~1, 129--146.

\bibitem{KR}
O. Kov\'a\v{c}ik and J. R\'akosn{\'i}k, \emph{On spaces
${L}^{p(x)}$ and
  ${W}^{k,p(x)}$}, Czechoslovak Math. J \textbf{41} (1991), 592--618.

\bibitem{ROF}
L. I. {Rudin}, S. Osher  and E. Fatemi, \emph{Nonlinear total
variation
  based noise removal algorithms}, Physica D Nonlinear Phenomena \textbf{60}
  (1992), 259--268.

\bibitem{R}
M. R{\ocirc{u}}{\v{z}}i{\v{c}}ka, \emph{Electrorheological fluids:
  modeling and mathematical theory}, Lecture Notes in Mathematics, vol. 1748,
  Springer-Verlag, Berlin, 2000.

\bibitem{Sam1}
S. Samko, \emph{Denseness of {$C\sp \infty\sb 0(\bold R\sp N)$} in
the
  generalized {S}obolev spaces {$W\sp {M,P(X)}(\bold R\sp N)$}}, Direct and
  inverse problems of mathematical physics ({N}ewark, {DE}, 1997), Int. Soc.
  Anal. Appl. Comput., vol.~5, Kluwer Acad. Publ., Dordrecht, 2000,
  pp.~333--342.

\bibitem{Sh}
D. F. Shanno, \emph{Conditioning of quasi-newton methods for
function
  minimization}, Mathematics of Computing \textbf{24} (1970), 647--656.

\bibitem{Z}
V. V. Zhikov, \emph{Averaging of functionals of the calculus of
variations and
  elasticity theory}, Izv. Akad. Nauk SSSR Ser. Mat. \textbf{50} (1986), no.~4,
  675--710, 877.

\end{thebibliography}

\end{document}